\documentclass{article}[12pt]
\usepackage{hyperref}
\usepackage{amsmath,amssymb,latexsym,amsthm}
\usepackage{url}
\usepackage{graphicx,color}
\usepackage{mathdots}
\usepackage{algorithm}
\usepackage{algorithmic}
\usepackage[margin=1.2in]{geometry}
\usepackage{caption}
\newtheorem{thm}{Theorem}
\newtheorem{lemma}[thm]{Lemma}
\newtheorem{prop}[thm]{Proposition}

\def\argmin{\text{arg}\,\text{min}}

\def\vspan{\text{span}\,}

\graphicspath{{figures/}}

\title{A Non-Iterative Reconstruction Algorithm for the Acoustic Inverse Boundary Value Problem}
\author{Tianyu Yang and Yang Yang}
\date{}

\begin{document}

\maketitle

\abstract{We present a non-iterative algorithm to reconstruct the isotropic acoustic wave speed from the measurement of the Neumann-to-Dirichlet map. The algorithm is designed based on the boundary control method and involves only computations that are stable. We prove the convergence of the algorithm and present its numerical implementation. The effectiveness of the algorithm is validated on both constant speed and variable speed, 
with full and partial boundary measurement as well as different levels of noise.}

\section{Introduction}

This paper concerns numerical reconstruction of an isotropic wave speed in the inverse boundary value problem (IBVP) for the acoustic wave equation. Specifically, let $T>0$ be a constant and $\Omega\subset\mathbb{R}^n$ be a bounded domain with smooth boundary $\partial\Omega$. Consider the initial-boundary value problem for the acoustic wave equation:
\begin{equation} \label{eq:ibvp}
\left\{
\begin{array}{rcl}
\partial^2_t u(t,x) - c^2(x) \Delta u(t,x) & = & 0, \quad\quad\quad \text{ in } (0,2T) \times \Omega \\
 \partial_\nu u  & = & f, \quad\quad\quad \text{ on } (0,2T) \times \partial\Omega \\ 
 u(0,x) = \partial_t u(0,x) & = & 0 \quad\quad\quad\quad x \in \Omega.
\end{array}
\right.
\end{equation}
Here $c(x) \in C^\infty(\Omega)$ is a smooth wave speed bounded away from $0$ and $\infty$.
Denote the solution by $u(t,x)=u^f(t,x)$.

Given $f\in C^\infty_c((0,2T)\times\partial\Omega)$, the well-posedness of this problem is ensured by the standard theory for second order hyperbolic partial differential equations. Define the Neumann-to-Dirichlet(ND) map:
\begin{equation} \label{eq:NDmap}
\Lambda_c f := u^f|_{(0,2T)\times \partial\Omega}.
\end{equation}
The IBVP for the acoustic wave equation aims to recover the wave speed $c(x)$ from the knowledge of the ND map $\Lambda_c$.

This inverse problem lies at the core of many imaging technologies. An important example is the \textit{Ultra-Sound Computed Tomography} (USCT). In USCT, a point-like ultrasound source emits an acoustic pulse from a known location outside the tissue. The acoustic wave travels through the tissue and the resulting wave field is recorded by a collection of surrounding ultrasonic transducers. This process is repeated many times for plenty of emitter locations, see Figure~\ref{fig:USCT} for an illustration with $M$ transducers, The goal of USCT is to reconstruct the acoustic wave speed everywhere inside the tissue. Similar data acquisition scheme occurs in seismic tomography, where one attempts to recover the
underground wave speed to locate oil reservoirs. 
In the continuous formulation of USCT and seismic tomography, the measurement is the boundary values of the Green's function. However, it is well known~\cite{nachman1988reconstructions} that such data is equivalent to knowledge of the ND map $\Lambda_c$ under mild assumptions.
\begin{figure}[!htb] 
\begin{center} 
\includegraphics[width=0.5\textwidth]{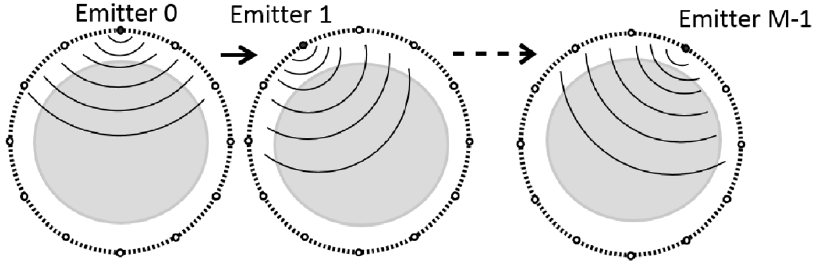}
\caption{Data acquisition scheme in USCT~\cite{matthews2017joint}} \label{fig:USCT}
\end{center}
\end{figure}

The IBVP for the acoustic wave equation has been extensively studied in the literature.
Among them, Belishev~\cite{belishev1988approach} proved that $c$ is uniquely determined using the boundary control (BC) method combined with Tataru's unique continuation result~\cite{Tataru95unique}. The result has since been greatly extended to wave equations with lower order terms on Riemannian manifolds with boundary~\cite{belishev1992reconstruction, eskin2006new, eskin2007inverse, eskin2008inverse, eskin2017inverse, feizmohammadi2019recovery, hu2017determination, isakov1992stability, kian2016recovery, kian2017recovery, kurylev2000hyperbolic, kurylev2018inverse, lassas2014inverse, ramm1991property, salazar2013determination, stefanov1989uniqueness}. Stability estimates have been obtained in~\cite{aicha2015stability, bao2014sensitivity, bellassoued2017stable, bellassoued2010stability, bosi2017reconstruction, liu2016lipschitz, montalto2014stable, stefanov1998stability, stefanov2005stable, stefanov2018lorentzian}.

The BC method has been numerically implemented to reconstruct the wave speed~\cite{belishev1999dynamical, belishev2016numerical, pestov2010numerical}. The implementation typically involves solving a control problem. This is achieved in~\cite{belishev1999dynamical, belishev2016numerical} using the so-called wave bases, and in~\cite{bingham2008iterative, de2018recovery} using the regularized optimization.
In the $1+1$ dimension, a discrete regularization strategy is developed in~\cite{korpela2018discrete} to recover $c$ from a single pulse-like source.
A variant of the BC method has also been applied to detect blockage in networks~\cite{blaasten2019blockage}.

In this paper, we develop a BC-based algorithm to reconstruct the wave speed. 
The derivation is inspired by the theoretical proofs in~\cite{liu2016lipschitz}, see also~\cite{pestov1999reconstruction}.
The algorithm has several favorable features from the computational viewpoint:
(1) The algorithm is direct.
Conventional computational approaches to recover $c$ relies on minimization of a data misfit functional through iterations. 
These approaches suffer from local minima, where gradient-descent-based iterations are trapped thus fail to give the true solution to the imaging problem.
An example is the cycle-skipping effect in the full waveform inversion. In contrast, a BC-based method solves directly for the solution and involves no iteration. 
(2) The algorithm converges globally to the true speed. This is again in contrast to iterative algorithms, which converge to the global minimum only when the initial guess is sufficiently accurate. A resulting prospect is that our algorithm could be used to provide a reliable initial guess for iterative methods.
(3) The algorithm involves only computations that are stable. Following the idea in~\cite{liu2016lipschitz}, one can show that the algorithm is locally Lipschitz stable for
a low frequency component of $c^{-2}$. This is a distinction from the previous BC method in~\cite{de2018recovery}.
(4) The algorithm is robust to random noise. 
The derivation reveals that the ND map is naturally followed by a low-pass filter in the assembly of the connecting operator (see~\eqref{eq:K}). This filter helps remove high-frequency content in the ND map, leading to robust reconstruction with respect to random noises.



The paper is organized as follows. In Section~\ref{sec:derivation}, we derive the reconstruction algorithm and the convergence result using the boundary control theory. In Section~\ref{sec:imple}, we elucidate our implementation of the algorithm using the finite difference scheme. Section~\ref{sec:examples} is devoted to numerical experiments, where the algorithm is evaluated on both constant speed and variable speed, with full and partial boundary measurement as well as different levels of noise.

\section{Derivation and Convergence} \label{sec:derivation}

We derive the reconstruction algorithm and show its convergence in this section. Given a function $u(t,x)$, we write $u(t)=u(t,\cdot)$ for the spatial part as a function of $x$.
Introduce the time reversal operator $R: L^2([0,T]\times\partial\Omega) \rightarrow L^2([0,T]\times\partial\Omega)$,
\begin{equation} \label{eq:R}
Ru(t,\cdot):=u(T-t,\cdot) ,\quad\quad 0<t<T;
\end{equation}
and the low-pass filter $J: L^2([0,2T]\times\partial\Omega) \rightarrow L^2([0,T]\times\partial\Omega)$
\begin{equation} \label{eq:J}
Jf(t,\cdot):=\frac{1}{2}\int^{2T-t}_t f(\tau,\cdot) \,d\tau,\quad\quad 0<t<T.
\end{equation}
We write $P_T: L^2((0,2T)\times\partial\Omega) \rightarrow L^2((0,T)\times\partial\Omega)$
for the orthogonal projection via restriction. Its adjoint operator $P^\ast_T: L^2((0,T)\times\partial\Omega) \rightarrow L^2((0,2T)\times\partial\Omega)$
is the extension by zero from $(0,T)$ to $(0,2T)$. Let $\mathcal{T}_D$ and $\mathcal{T}_N$ be the Dirichlet and Neumann trace operators respectively, that is,
$$
\mathcal{T}_D u(t,\cdot) = u(t,\cdot)|_{\partial\Omega}, \quad\quad\quad \mathcal{T}_N u(t,\cdot) = \partial_\nu u(t,\cdot)|_{\partial\Omega}.
$$

\begin{lemma} \label{thm:id}
Let $u^f$ be the solution of \eqref{eq:ibvp} with $f\in C^\infty_c((0,2T)\times\partial\Omega)$. Suppose $v(t,x)\in C^\infty((0,2T)\times\Omega)$ satisfies the wave equation
$$
(\partial^2_t - c^2(x)\Delta) v(t,x) = 0, \quad\quad\quad \text{ in } (0,2T)\times\Omega
$$
Then
$$
(u^f(T), v(T))_{L^2(\Omega,c^{-2}dx)} = (P_T f,J\mathcal{T}_D v)_{L^2((0,T)\times\partial\Omega)} - (P_T(\Lambda_c f),J\mathcal{T}_N v)_{L^2((0,T)\times\partial\Omega)}.
$$
where $\nu$ is the unit outer normal vector field on $\partial\Omega$.
\end{lemma}
\begin{proof}
Define 
$$
I(t,s) := (u^f(t), v(s))_{L^2(\Omega,c^{-2}dx)}.
$$
We compute
\begin{align}
 & (\partial^2_t - \partial^2_s) I(t,s) \nonumber \\
= & (\Delta u^f(t), v(s))_{L^2(\Omega)} - (u^f(t),\Delta v(s))_{L^2(\Omega)} \nonumber \\
= & (f(t), \mathcal{T}_D v(s))_{L^2(\partial\Omega)} - (\Lambda_c f(t),\mathcal{T}_N v(s))_{L^2(\partial\Omega)}, \label{eq:id1}
\end{align}
where the last equality follows from integration by parts.
On the other hand, $I(0,s)=\partial_t I(0,s) = 0$ since $u^f(0,x)=\partial_t u^f(0,x) = 0$. Solve the inhomogeneous $1$D wave equation \eqref{eq:id1} together with these initial conditions to obtain 
\begin{align*}
I(T,T) & = \frac{1}{2} \int^T_0 \int^{2T-t}_{t} \left[ (f(t), \mathcal{T}_D v(\sigma))_{L^2(\partial\Omega)} - (\Lambda_c f(t),\mathcal{T}_N v(\sigma))_{L^2(\partial\Omega)} \right] \,d\sigma dt \vspace{1ex}\\
 & = \int^T_0  [ ( f(t) ,  \frac{1}{2}\int^{2T-t}_{t} \mathcal{T}_D v(\sigma) \,d\sigma)_{L^2(\partial\Omega)} - ( \Lambda_c f(t),  \frac{1}{2}\int^{2T-t}_{t} \mathcal{T}_N v(\sigma) \,d\sigma)_{L^2(\partial\Omega)} ] \,dt \vspace{1ex} \\
 & = (P_T f,J \mathcal{T}_D v)_{L^2((0,T)\times\partial\Omega)} - (P_T(\Lambda_c f),J \mathcal{T}_N v)_{L^2((0,T)\times\partial\Omega)}.
\end{align*}
\end{proof}

We will use the lemma to derive two results. The first is the Blagove\u{s}\u{c}enski\u{ı}’s identity. 
To this end, denote by $\Lambda_{c,T}$ the ND map defined as in \eqref{eq:ibvp} \eqref{eq:NDmap} yet with $2T$ replaced by $T$. It can be easily verified from integration by parts that its adjoint operator (with respect to the inner product in $L^2((0,T)\times\partial\Omega)$) is $\Lambda^\ast_{c,T} = R \Lambda_{c,T} R$ where $R$ is the time reversal operator \eqref{eq:R}.

Introduce the \textit{connecting operator}
\begin{equation} \label{eq:K}
K:= J \Lambda_c P^\ast_T - R \Lambda_{c,T} R J P^\ast_T.
\end{equation}
The operator $K$ connects inner-products between waves in the interior to measurements on the boundary. It is the principal object of the boundary control method \cite{belishev2007recent}.
Moreover, $K$ is a compact operator since 
$\Lambda_{c,T}: L^2((0,T)\times\partial\Omega)\rightarrow H^{2/3}((0,T)\times\partial\Omega)$ is smoothing, see~\cite{tataru1998regularity}.

The Blagove\u{s}\u{c}enski\u{ı}’s identity we will establish is slightly different from its original form~\cite{blagoveshchenskii1967inverse}. Instead, it is a reformulation that has been previously used in \cite{bingham2008iterative, oksanen2013solving, de2018exact}.

\begin{prop} \label{thm:waveinner}
Let $u^f, u^h$ be the solutions of \eqref{eq:ibvp} with Neumann traces $f, h \in L^2((0,T)\times\partial\Omega)$, respectively. Then
\begin{equation} \label{eq:ufuh}
(u^f(T), u^h(T))_{L^2(\Omega,c^{-2}dx)} = (f,Kh)_{L^2((0,T)\times\partial\Omega)} = (Kf,h)_{L^2((0,T)\times\partial\Omega)}.
\end{equation}
In particular if $h=f$, one has
\begin{equation} \label{eq:ufnorm}
\|u^f(T)\|^2_{L^2(\Omega,c^{-2}dx)} = (f,Kf)_{L^2((0,T)\times\partial\Omega)} = (Kf,f)_{L^2((0,T)\times\partial\Omega)}.
\end{equation}
\end{prop}
\begin{proof}
We first prove this for $f, h \in C^\infty_c((0,T)\times\partial\Omega)$. Apply Lemma \ref{thm:id} to $u^f$ and $v=u^h$ and notice that $\mathcal{T}_D u^h = \Lambda_c P^\ast_T h$ and $\mathcal{T}_N u^h = P^\ast_T h$. One has
\begin{align*}
(u^f(T), u^h(T))_{L^2(\Omega,c^{-2}dx)} & = (P_T f,J \Lambda_c P^\ast_T h)_{L^2((0,T)\times\partial\Omega)} - (P_T(\Lambda_c f),J P^\ast_T h)_{L^2((0,T)\times\partial\Omega)} \\
 & = (f,J \Lambda_c P^\ast_T h)_{L^2((0,T)\times\partial\Omega)} - (\Lambda_{c,T} f, J P^\ast_T h)_{L^2((0,T)\times\partial\Omega)} \\
 & = (f,J \Lambda_c P^\ast_T h)_{L^2((0,T)\times\partial\Omega)} - (f, R \Lambda_{c,T} R J P^\ast_T h)_{L^2((0,T)\times\partial\Omega)} \\
 & = (f,Kh)_{L^2((0,T)\times\partial\Omega)}
\end{align*}
where we have used that $P_T(\Lambda_c f) = \Lambda_{c,T}f$ and that $\Lambda^\ast_{c,T} = R \Lambda_{c,T} R$ in $L^2((0,T)\times\partial\Omega)$. This establishes the first equality in \eqref{eq:ufuh}. Interchanging $f$ and $h$ yields the second equality in \eqref{eq:ufuh}.

For general $f, h \in L^2((0,T)\times\partial\Omega)$, simply notice that $K$ is a continuous operator and that compactly supported smooth functions are dense in $L^2$. The proof is completed.
\end{proof}


The Blagove\u{s}\u{c}enski\u{ı}’s identity relates inner products of waves to boundary measurement. Next, we derive an identity that allows computation of inner products between waves and harmonic functions from boundary data. 
We introduce another operator $B$ that is critical for our reconstruction: 
\begin{equation} \label{eq:B}
B := J \mathcal{T}_D - R \Lambda_{c,T} R J \mathcal{T}_N.
\end{equation}
\begin{prop} \label{thm:waveharm}
Let $u^f$ be the solutions of \eqref{eq:ibvp} with Neumann traces $f\in L^2((0,T)\times\partial\Omega)$. For any harmonic function $\phi\in C^\infty(\Omega)$, one has
$$
(u^f(T), \phi)_{L^2(\Omega,c^{-2}dx)} = (f,B\phi)_{L^2((0,T)\times\partial\Omega)}.
$$
\end{prop}

\begin{proof}
We only need to prove this for $f\in C^\infty_c((0,T)\times\partial\Omega)$ by the continuity of $B$ and density of compactly supported functions in $L^2$. Apply Lemma \ref{thm:id} to $u^f$ and $v=\phi$ (since any harmonic function is a time-independent solution of the acoustic wave equation). 
One has
\begin{align*}
(u^f(T), \phi)_{L^2(\Omega,c^{-2}dx)} = & (f,J \mathcal{T}_D\phi)_{L^2((0,T)\times\partial\Omega)} - (P_T(\Lambda_c f),J \mathcal{T}_N \phi )_{L^2((0,T)\times\partial\Omega)} \\
 = & (f,J \mathcal{T}_D \phi)_{L^2((0,T)\times\partial\Omega)} - (\Lambda_{c,T} f, J \mathcal{T}_N \phi)_{L^2((0,T)\times\partial\Omega)} \\
 = & (f,J \mathcal{T}_D \phi)_{L^2((0,T)\times\partial\Omega)} - (f, R \Lambda_{c,T} R J \mathcal{T}_N \phi)_{L^2((0,T)\times\partial\Omega)}.
\end{align*}
\end{proof}

Proposition \ref{thm:waveharm} suggests a way to reconstruct the wave speed $c$ from the ND map $\Lambda_c$: if for any harmonic function $\psi$, one can find an explicit sequence $f_\alpha$ such that $u^{f_\alpha}(T) \rightarrow \psi$ as $\alpha\rightarrow 0$ in $L^2(\Omega, c^{-2}dx)$, then
\begin{equation} \label{eq:limit}
(\psi, \phi)_{L^2(\Omega,c^{-2}dx)} = \lim_{\alpha\rightarrow  0}(u^{f_\alpha}(T), \phi)_{L^2(\Omega,c^{-2}dx)} 
= \lim_{\alpha\rightarrow 0}(f_\alpha,B\phi)_{L^2((0,T)\times\partial\Omega)}.
\end{equation}
The right hand side can be computed from $\Lambda_c$, see \eqref{eq:B}. Thus the integral
\begin{equation} \label{eq:innerproduct}
(\psi, \phi)_{L^2(\Omega,c^{-2}dx)} = \int_\Omega \psi \phi \, c^{-2}(x) \,dx 
\end{equation}
is known for all harmonic functions $\psi$ and $\phi$. For any fixed vectors $\xi,\eta\in\mathbb{R}^n$ with $|\xi|=|\eta|$ and $\xi\perp\eta$, choose the complex harmonic functions
\begin{equation} \label{eq:phipsi}
\psi(x):=e^{\frac{i}{2}(-\xi+i\eta)\cdot x}, \quad\quad \phi(x):=e^{\frac{i}{2}(-\xi-i\eta)\cdot x}.
\end{equation}
Then $\psi\phi = e^{i\xi\cdot x}$ and one recovers $\mathcal{F}(c^{-2})$ -- the Fourier transform of $c^{-2}$ -- by varying $\xi$.  This reconstructs $c$.

It remains to construct an explicit sequence $f_\alpha$ such that $u^{f_\alpha}(T) \rightarrow \psi$ in $L^2(\Omega, c^{-2}dx)$as $\alpha\rightarrow 0$. We will adopt Tikhonov regularization for the construction. Before that, we record a lemma that will be used in the subsequent analysis.
\begin{lemma}[{\cite[Lemma 1]{oksanen2013solving}}] \label{thm:projection}
Let $A:X\rightarrow Y$ be a bounded linear operator between two Hilbert spaces $X$ and $Y$. For any $y\in Y$, let $\alpha>0$ be a constant and $x_\alpha := (A^\ast A + \alpha)^{-1} A^\ast y$. Then 
$$
A x_\alpha \rightarrow P_{\overline{Ran(A)}}y \quad\quad \text{ as } \alpha\rightarrow 0
$$
where $P_{\overline{Ran(A)}}y$ denotes the orthogonal projection of $y$ onto the closure of the range of $A$.
\end{lemma}

Next, we introduce the \textit{control operator} 
$$
Wf := u^f(T).
$$
where $u^f$ is the solution of \eqref{eq:ibvp}. According to~\cite{lasiecka1991regularity},
$W: L^2((0,T)\times\partial\Omega) \rightarrow L^2(\Omega)$ is a bounded linear operator.
Moreover, Tataru's theorem in \cite{tataru1995unique} implies that $W$ has dense range in $L^2(\Omega)$.
It follows from Proposition~\ref{thm:waveinner} that $K=W^* W$. It is also easy to verify that $W^\ast \psi = B\psi$ for any harmonic function $\psi$.


\begin{prop} \label{thm:Tik}
For any harmonic function $\psi$, the following minimization problem with parameter $\alpha>0$:
$$
f_\alpha := \argmin_{f} \|Wf-\psi\|^2_{L^2(\Omega,c^{-2}dx)} + \alpha \|f\|^2_{L^2(0,T)\times\partial\Omega}.
$$
has a unique solution $f_\alpha \in L^2((0,T)\times\partial\Omega)$. This solution satisfies the linear equation 
\begin{equation} \label{eq:normal}
(K+\alpha) f_\alpha = B\psi.
\end{equation}
Moreover, $u^{f_\alpha}(T) \rightarrow \psi$ as $\alpha\rightarrow 0$ in $L^2(\Omega, c^{-2}dx)$. 
\end{prop}
\begin{proof}
The functional to be minimized is
$$
F_\alpha(f) := \|Wf-\psi\|^2_{L^2(\Omega,c^{-2}dx)} + \alpha \|f\|^2_{L^2((0,T)\times\partial\Omega)}.
$$
As $W: L^2((0,T)\times\partial\Omega) \rightarrow L^2(\Omega)$ is bounded and linear, \cite[Theorem 2.11]{kirsch2011introduction} claims that $F_\alpha$ has a unique minimizer, named $f_\alpha$, in $L^2((0,T)\times\partial\Omega)$.

To derive the normal equation the minimizer obeys, we rewrite 
\begin{align*}
F_\alpha(f) = & \|u^f(T)-\psi\|^2_{L^2(\Omega,c^{-2}dx)} + \alpha \|f\|^2_{L^2((0,T)\times\partial\Omega)} \\
 = & \|u^f(T)\|^2_{L^2(\Omega,c^{-2}dx)} - 2(u^f(T),\psi)_{L^2(\Omega,c^{-2}dx)} + \|\psi\|^2_{L^2(\Omega,c^{-2}dx)} + \alpha \|f\|^2_{L^2(0,T)\times\partial\Omega} \\
 = & (f,Kf)_{L^2((0,T)\times\partial\Omega)} - 2 (f, B\psi)_{L^2((0,T)\times\partial\Omega)} + \|\psi\|^2_{L^2(\Omega,c^{-2}dx)} + \alpha \|f\|^2_{L^2((0,T)\times\partial\Omega)} \\
 = &  (f,(K+\alpha)f)_{L^2((0,T)\times\partial\Omega)} - 2 (f, B\psi)_{L^2((0,T)\times\partial\Omega)} + \|\psi\|^2_{L^2(\Omega,c^{-2}dx)}
\end{align*}
where we have used Proposition \ref{thm:waveinner} and Proposition \ref{thm:waveharm} in the second but last line. This is a bilinear form of $f$ whose Frech\'{e}t derivative is
$$
F'(f) = 2(K+\alpha)f - 2B\psi.   
$$
The minimizer satisfies $F'(f_\alpha)=0$, hence \eqref{eq:normal}.
%

Finally, since $K=W^\ast W$ and $B \psi = W^\ast \psi$ (see the remark before Proposition~\ref{thm:Tik}), We conclude from Lemma~\ref{thm:projection} that $W f_\alpha \rightarrow P_{\overline{Ran(W)}}\psi$ in $L^2(\Omega,c^{-2}dx)$ as $\alpha\rightarrow 0$. Tataru's theorem \cite{tataru1995unique} claims that the range of $W$ is dense in $L^2(\Omega)$, hence $P_{\overline{Ran(W)}}\psi = \psi$. 
\end{proof}

\bigskip

Summarizing the discussion in this section, we have proved global convergence of the following reconstruction algorithm:

\begin{algorithm}[!hbt]
	\caption{(Non-Iterative Reconstruction Algorithm for Acoustic IBVP). \\
	\textit{Input: low-pass filter $J$, time-reversal operator $R$, projection operator $P_T$, ND map $\Lambda_c$}\\
	\textit{Output: wave speed $c$}
	} \label{alg:main}
	\begin{algorithmic}[1]
		\STATE Assemble the connecting operator $K= J \Lambda_c P^\ast_T - R \Lambda_{c,T} R J P^\ast_T$ (see~\eqref{eq:K}). 
		\STATE Assemble the operator $B = J \mathcal{T}_D - R \Lambda_{c,T} R J \mathcal{T}_N$ (see~\eqref{eq:B}).
	    \STATE Construct the harmonic function  $\psi(x)=e^{\frac{i}{2}(-\xi+i\eta)\cdot x}$ (see~\eqref{eq:phipsi}) and solve the linear system $(K+\alpha)f_\alpha = B\psi$, (see~\eqref{eq:normal}).
	    \STATE Construct the harmonic function $\phi(x):=e^{\frac{i}{2}(-\xi-i\eta)\cdot x}$ (see~\eqref{eq:phipsi}) and compute the Fourier projection
	    $$
	    \int_\Omega e^{-i\xi\cdot x} c^{-2}(x) \,dx = \lim_{\alpha\rightarrow 0}(f_\alpha,B\phi)_{L^2((0,T)\times\partial\Omega)}
	    $$
	    through the limiting process, (see~\eqref{eq:limit}).
        \STATE Repeat the above steps with various $\xi$ to recover the Fourier transform $\mathcal{F}(c^{-2})$.
	    \STATE Invert the Fourier transform to recover $c^{-2}$, and eventually $c$.
	\end{algorithmic}
\end{algorithm}








\section{Algorithm Implementation} \label{sec:imple}

In this section, we provide details of our implementation of the algorithm using finite difference discretization.

\subsection{Forward Simulation.}

\textbf{Computational Domain and Grid.}
We take the computational domain $\Omega = [-1,1]\times [-1,1]$, and write $t \in [0,T]$ for the temporal variable and $(x,y)\in\Omega$ for the two spatial coordinates, respectively.
Let $0 = t_0 \leq t_1 \leq \dots \leq t_L = T$ be a partition of the interval $[0,T]$ with uniform spacing $\Delta t = \frac{T}{L}$. Let $-1 = x_0 \leq x_1 \leq \dots \leq x_I = 1$ be a partition of the interval $[-1,1]$ with uniform spacing $\Delta x = \frac{2}{I}$. 
Then the temporal grid points are $t_l = l \Delta t \in [0,T]$, $l=0,1,\dots, L$. The spatial grid points are $(x_i,y_j)\in\Omega$ with $x_i = x_0 + i \Delta x$ and $y_j = x_0 + j \Delta x$, $i,j = 0,1,\dots, I$. 
The total grid size is $(L+1)\times (I+1) \times (I+1)$.

We denote the collection of interior grid points by
$$
IGP := 
\{(t_l,x_i,y_j): -1< x_i < 1, \;  1< y_j < 1, \; i,j = 1, \dots, I, \quad
l = 0,1,\dots, L\},
$$
and the collection of boundary grid points by
$$
BGP := \{(t_l,x_i,y_j): |x_i|=1, \; |y_j|=1, \; i,j = 1, \dots, I, \quad
l = 0,1,\dots, L\}.
$$
Let $u$ be the solution to the initial-boundary value problem~\eqref{eq:ibvp}. The values of $u$ on the grid points are denoted by 
$$
u^l_{ij} := u(t_l, x_i, y_j), \quad\quad\quad 
l=0,1,\dots, L, \quad i,j=0,1,\dots, I.
$$

\textbf{Forward Solver.} We solve the inverse boundary value problem~\eqref{eq:ibvp} by discretizing the acoustic wave equation using the second-order central difference scheme. 
For the interior grid points, the second order temporal and spatial derivatives are approximated as
$$
\partial^2_t u (t_l,x_i,y_j) \approx
\frac{u^{l-1}_{i,j} + u^{l+1}_{i,j} - 2 u^{l}_{i,j}}{\Delta t^2};
$$
$$
\Delta u (t_l,x_i,y_j) \approx \frac{u^l_{i-1,j}+u^l_{i+1,j}+u^l_{i,j-1}+u^l_{i,j+1}-4u^l_{i,j}}{\Delta x^2},
$$
thus we can update the interior grid points by
$$
u^{l+1}_{i,j}=2u^l_{i,j}-u^{l-1}_{i,j}+c^2(x_i,y_j)\frac{\Delta t^2}{\Delta x^2}[u^l_{i-1,j}+u^l_{i+1,j}+u^l_{i,j-1}+u^l_{i,j+1}-4u^l_{i,j}],
$$
here we set $u^1_{i,j}=u^{-1}_{i,j}$ for the initial condition $\partial_tu|_{t=0}=0.$
For the boundary grid points, the boundary normal derivative (i.e, Neumann data) is computed using the forward/backward finite difference approximation with a second-order accuracy. For instance, for $i=0$,
$$
\partial_\nu u(t_l, x_0, y_j) \approx
- \frac{3 u^l_{0,j} - 4 u^l_{1,j} + u^l_{2,j} }{2\Delta x}.
$$
The restriction $\Delta t=\frac{\sqrt{2}}{2c_{max}}\Delta x$ is imposed to fulfill the Courant–Friedrichs–Lewy (CFL) condition.
The forward simulation is implemented on the spatial grid with $I=100$. This grid is finer than the one used in the reconstruction to avoid the inverse crime.


\textbf{Assembly of the Discrete Neumann-to-Dirichlet Map.}
The spatial boundary $\partial\Omega$ consists of $4I$ boundary grid points, thus the temporal boundary $[0,T]\times\partial\Omega$ contains $4I(L+1)$ boundary grid points in total. These boundary grid points are ordered in the lexicographical order to form a column vector, that is, a boundary grid point $(t_l,x_i,y_j)$ is ahead of another $(t_{l'},x_{i'},y_{j'})$ if and only if (1) $l<l'$; or (2) $l=l'$ and $i<i'$; or (3) $l=l',i=i',j< j'$. In this way, the discretized ND map is a $4I(L+1)\times 4I(L+1)$ square matrix, denoted by $[\Lambda_c] \in\mathbb{R}^{4I(L+1)\times 4I(L+1)}$. In order to find the matrix representation $[\Lambda_c]$, we place a unit source $f_{lij}$ on each boundary grid point $(t_l,x_i,y_j)\in BGP$ as the Neumann data and utilize the forward solver to obtain the resulting Dirichlet data on all the boundary grid points. Here $f_{lij}$ takes the value $1$ on $(t_l,x_i,y_j)$ and $0$ on all the other boundary grid points, see Figure~\ref{fig:structure_NDmap} for an image of the ND map.




\begin{figure}[!htb]
    \centering
    \includegraphics[width=0.5\textwidth]{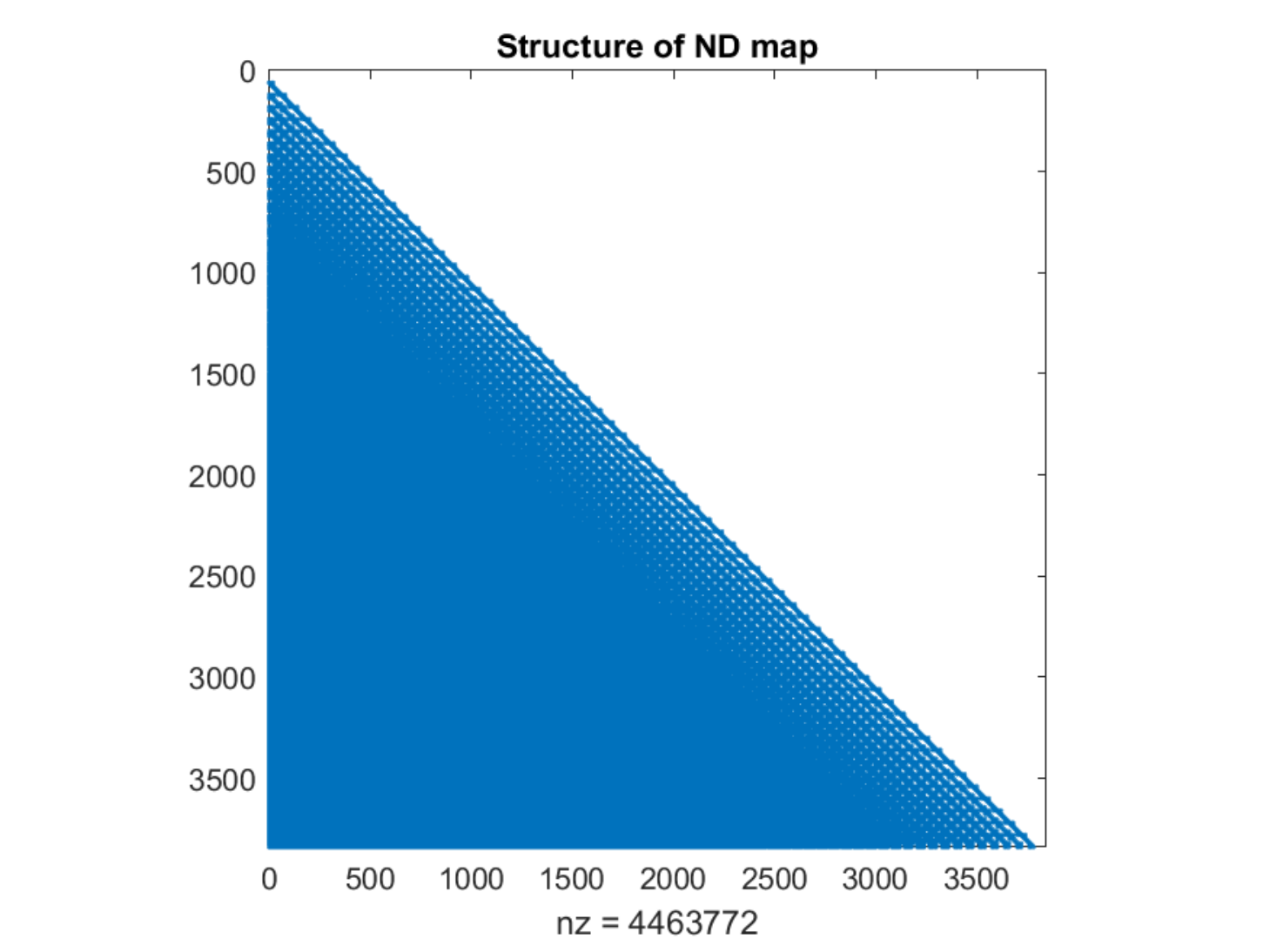}
    \caption{The structure of $[\Lambda_c]$ with $I=15$ and $L=63$. nz is the number of nonzero elements in the matrix.}
    \label{fig:structure_NDmap}
\end{figure}



\subsection{Reconstruction Algorithm.}


\textbf{Discretization of the Connecting Operator $K$.}
First, we discretize the operators in the definition of $K$, see~\eqref{eq:K}. For the filtering operator $J$, the integral in its definition~\eqref{eq:J} is discretized using the boundary grid points and trapezoidal rule. 
$$
    \int^{2T-t_l}_{t_l} f(\tau,\cdot) \,d\tau\approx\sum_{k=l}^{L-l-1}\frac{f(t_k,\cdot)+f(t_{k+1},\cdot)}{2}\Delta t.
$$

With the arrangement of the boundary grid points in the lexicographical order, the boundary vector consists of $L+1$ small vectors of length $4I$, where each small vector corresponds to the spatial boundary points at the moment $t=t_l, l=0,\dots,L$. According to the trapezoidal rule, we obtain the matrix representation $[J] \in \mathbb{R}^{2I(K+1)\times 4I(K+1)}$. It can be partitioned into $\lceil\frac{L+1}{2}\rceil\times (L+1)$ blocks, where $\lceil\frac{L+1}{2}\rceil$ denotes the smallest integer no smaller than $\frac{L+1}{2}$, see Figure~\ref{fig:structure_operators}.
Each block is a $4I\times4I$ identity matrix $[I]$ multiplied by the coefficients of the trapezoidal integration formula.
Specifically, if $L$ is odd,
$$
[J]=\frac{\Delta t}{2}
\begin{pmatrix}
[I]&2[I]&2[I]&\dots&\dots&\dots&\dots&2[I]&2[I]&[I]\\
&[I]&2[I]&\dots&\dots&\dots&\dots&2[I]&[I]&\\
&&\ddots&\ddots&&&\iddots&\iddots&&\\
&&&[I]&2[I]&2[I]&[I]&&&\\
&&&&[I]&[I]&&&&
\end{pmatrix},
$$
If $L$ is even,
$$
[J]=\frac{\Delta t}{2}
\begin{pmatrix}
[I]&2[I]&2[I]&\dots&\dots&\dots&2[I]&2[I]&[I]\\
&[I]&2[I]&\dots&\dots&\dots&2[I]&[I]&\\
&&\ddots&\ddots&&\iddots&\iddots&&\\
&&&[I]&2[I]&[I]&&&\\
&&&&[O]&&&&
\end{pmatrix},
$$
where $[O]$ is the $4I\times4I$ zero matrix.

Likewise, the time-reversal operator $R$ defined in~\eqref{eq:R} and the restriction operator $P_T$ are discretized to obtain their discrete counterparts $[R] \in \mathbb{R}^{2I(L+1)\times 2I(L+1)}$ and $[P_T] \in \mathbb{R}^{2I(L+1)\times 4I(L+1)}$. Thanks to the lexicographical order of the boundary grid points, these matrices have block structures as well:
$[R]$ is a square anti-diagonal block matrix
where the blocks are $4I\times4I$ identity matrices, and $[P_T]$ is a rectangular matrix with $1$ on the main diagonal:
$$
[P_T]=
\begin{pmatrix}
[I]_{[4I\lceil\frac{L+1}{2}\rceil]\times[4I\lceil\frac{L+1}{2}\rceil]}&[O]
\end{pmatrix},
[R]=
\begin{pmatrix}
&&[I]\\
&\iddots&\\
[I]&&
\end{pmatrix}
$$

\begin{figure}[!htb]
    \centering
    \includegraphics[width=0.49\textwidth]{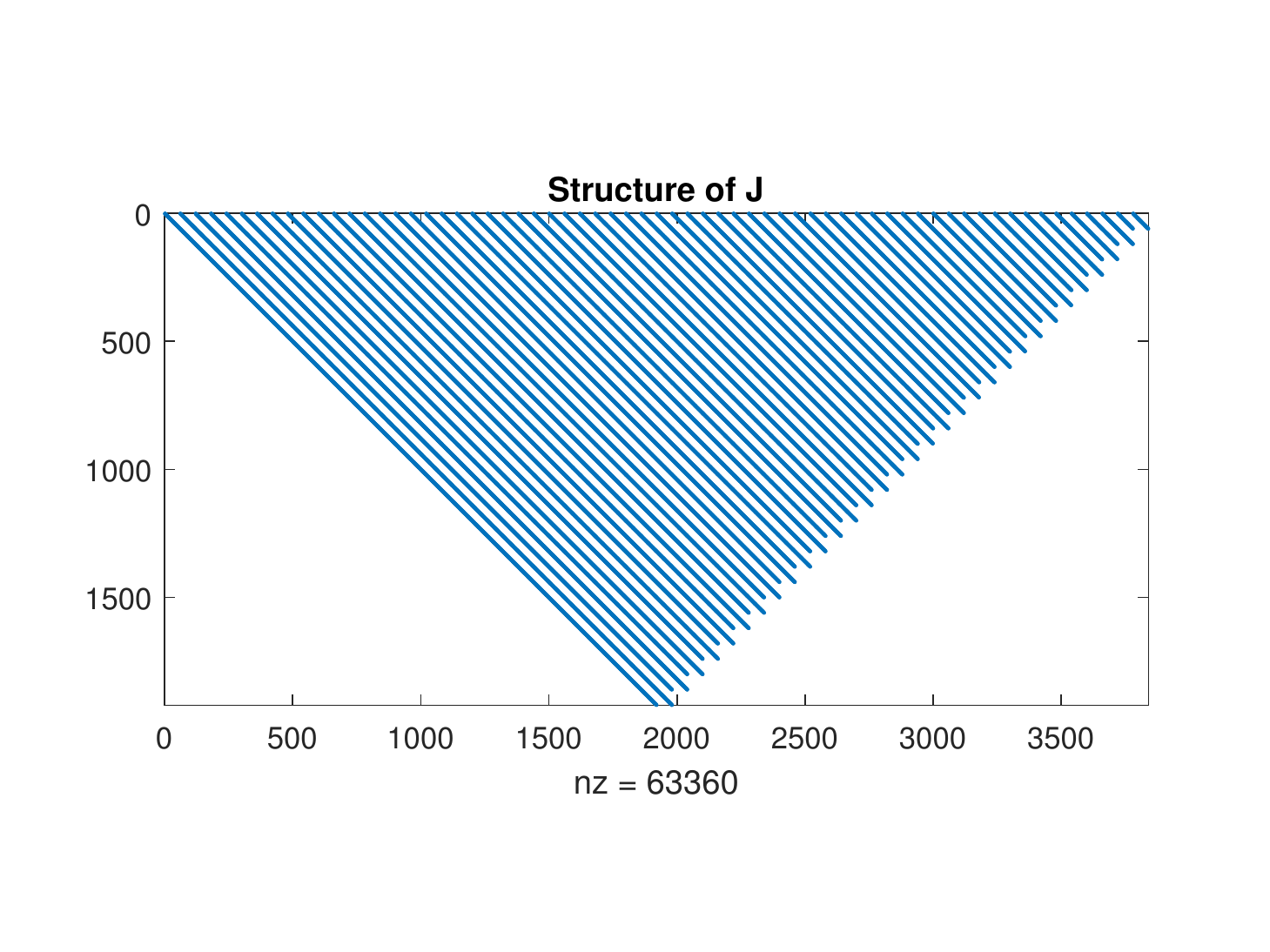}
    \includegraphics[width=0.49\textwidth]{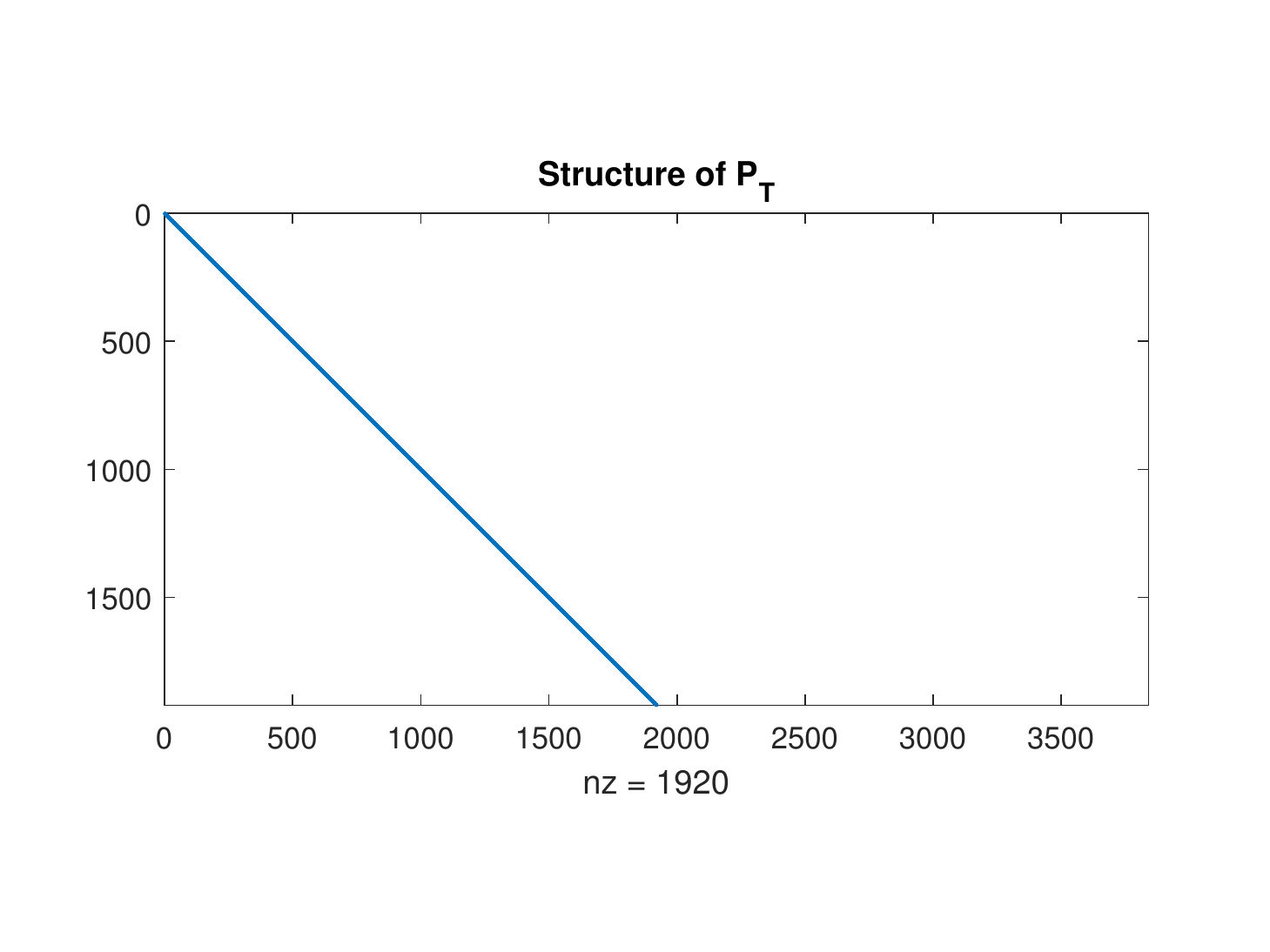}

    \includegraphics[width=0.49\textwidth]{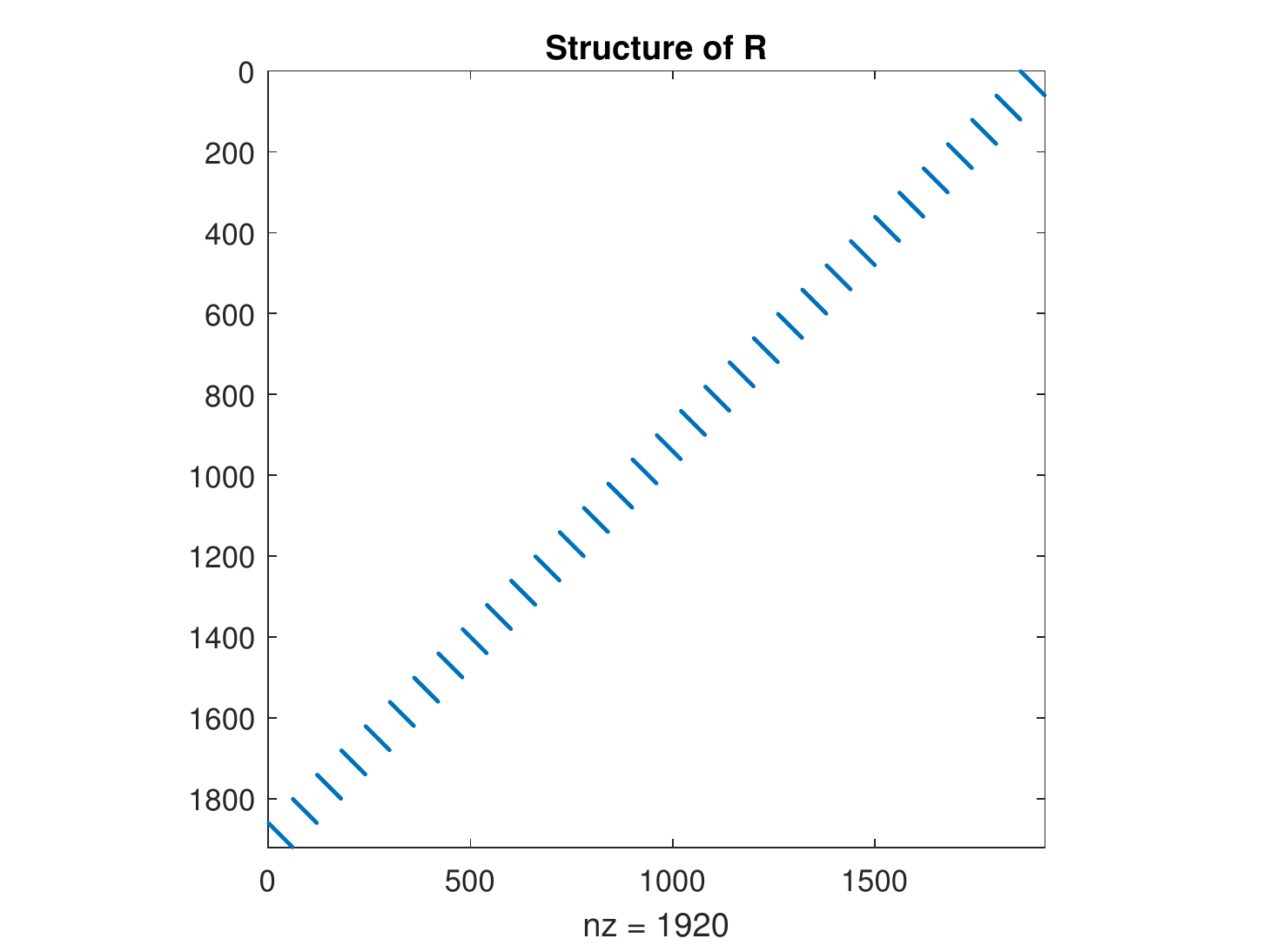}
    \includegraphics[width=0.49\textwidth]{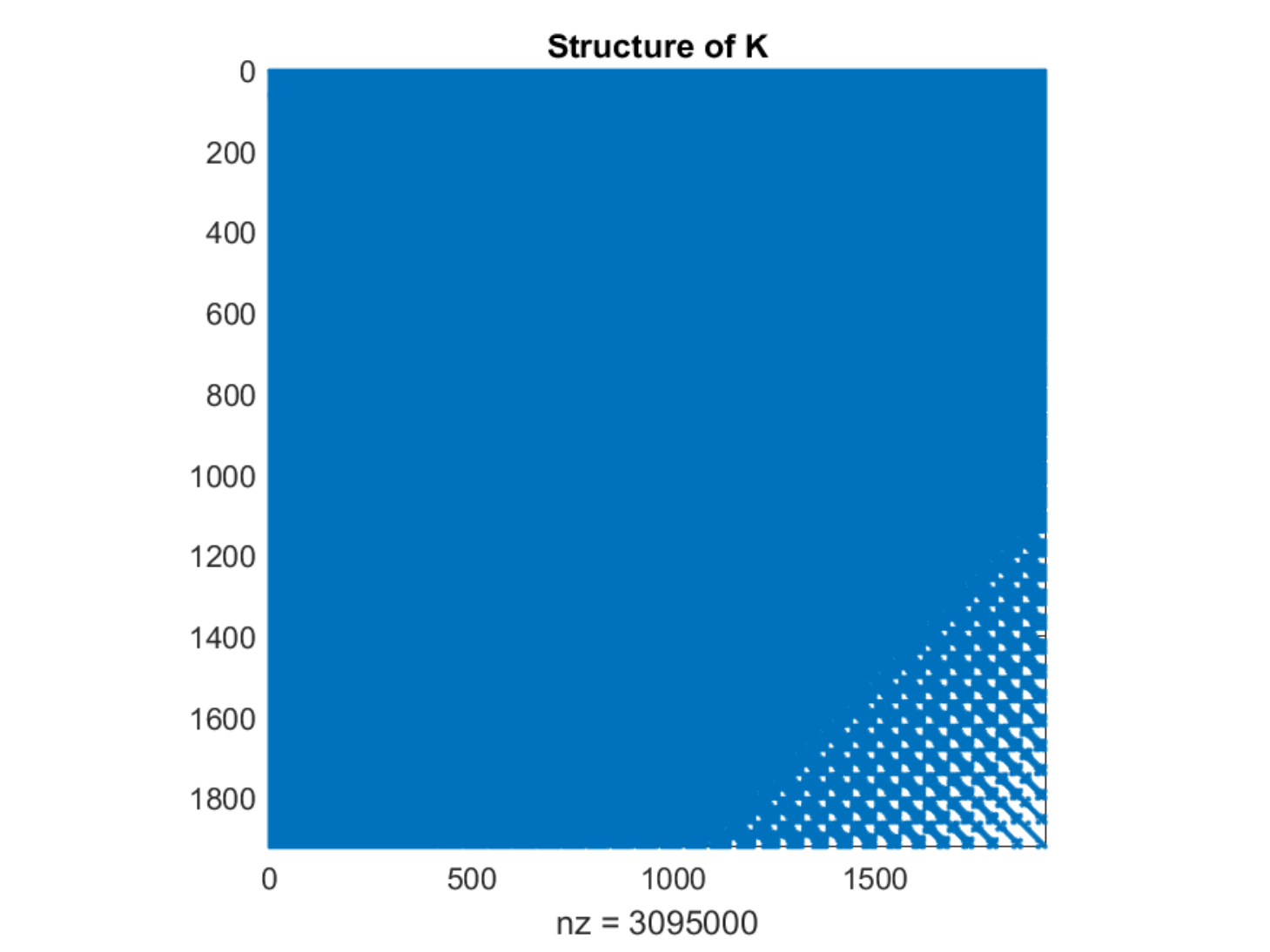}
    \caption{The structure of $[J],[P_T],[R],[K]$ with $I=15$ and $L=63$. nz is the number of nonzero elements in the matrix.}
    \label{fig:structure_operators}
\end{figure}

 \begin{figure}[!htb]
     \centering
     \includegraphics[width=0.49\textwidth]{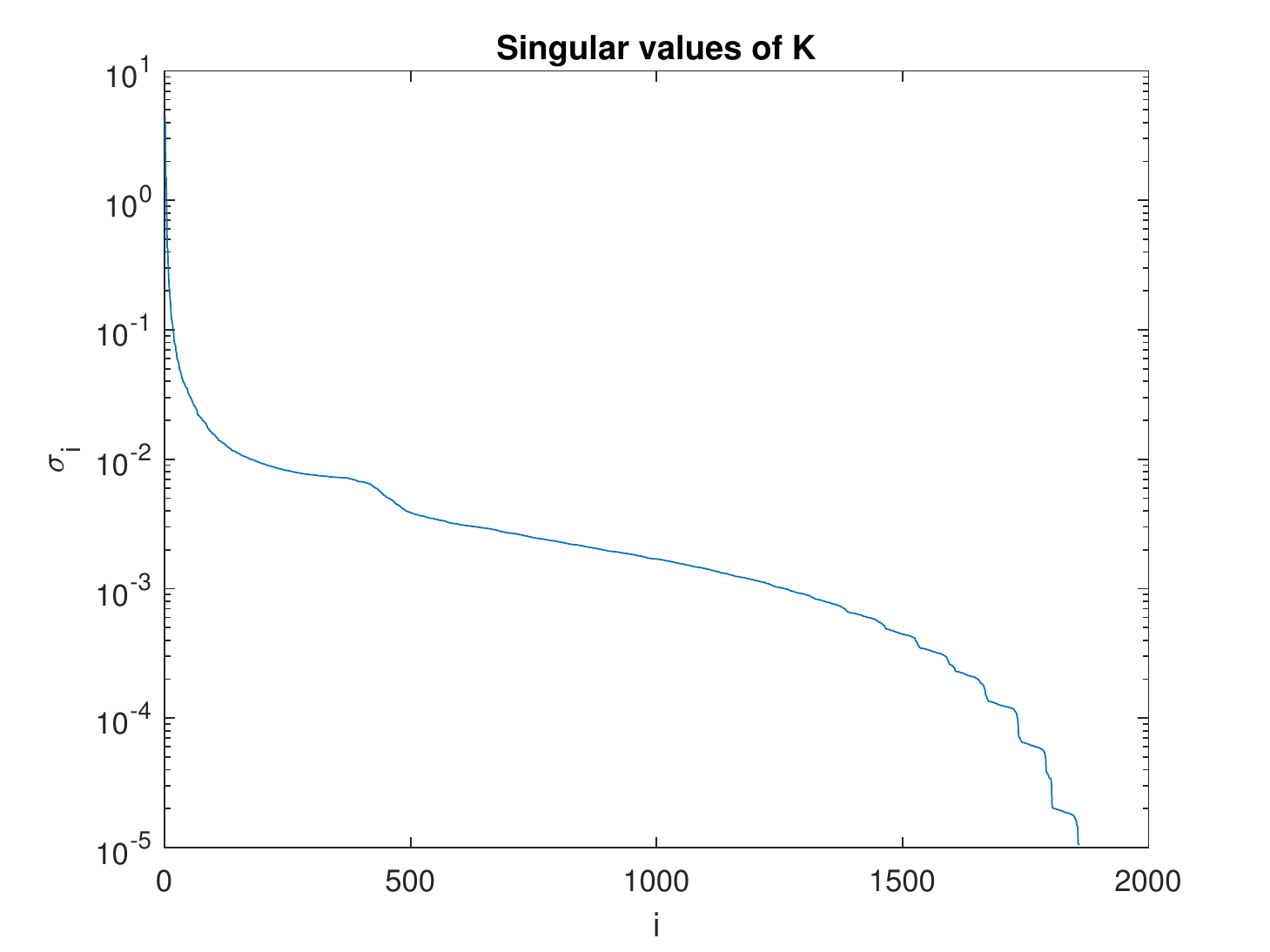}
     \caption{The singular values of $[K]$ with $I=15,L=63$. $[K]$ has 60 zero singular values.}
     \label{fig:singular_K}
 \end{figure}
The discretized adjoint $[P^\ast_T]$ is taken to be $[P_T]^t$, the transpose of $[P_T]$. Finally, the discretized $K$ is the following matrix product, according to~\eqref{eq:K}:
$$
[K] = [J] [\Lambda_c] [P_T]^t - [R] [\Lambda_{c,T}] [R] [J] [P_T]^t \quad\quad \in \mathbb{R}^{2I(L+1)\times 2I(L+1)}.
$$
In general, $[K]$ is not a sparse matrix.
An image of $[K]$ is illustrated in Figure~\ref{fig:structure_operators}. Since $K$ is a compact operator (see the remark below~\eqref{eq:K}), $[K]$ is ill-conditioned. Its singular values are plotted in Figure~\ref{fig:singular_K}.

\textbf{Discretization of the Operator $B$.}
With the aforementioned discretized operators, the discretized $B$ is naturally the following matrix product, according to~\eqref{eq:B}:
$$
[B] = [J] [\mathcal{T}_D] - [R] [\Lambda_{c,T}] [R] [J] [\mathcal{T}_N]  \quad\quad \in \mathbb{R}^{2I(L+1)\times 4I(L+1)}.
$$
Here the matrices $[\mathcal{T}_D], [\mathcal{T}_N] \in\mathbb{R}^{4I(L+1)\times 4I(L+1)}$ are of large size, thus their storage takes up lots of memory. However, observing that the operator $B$ is applied only to harmonic functions which are time-independent, we can reduce the cost of memory by first computing these matrices at a specific time, then shifting them to other times. Since the harmonic functions in our numerical experiments are all handcrafted, we can also compute their boundary values from the analytic expressions.

\textbf{Solving for $f_\alpha$.}
The next step is to solve for $[f_\alpha]$ from the discretized version of~\eqref{eq:normal}:
\begin{equation} \label{eq:discretenormal}
([K] + \alpha) [f_\alpha] = [B] [\psi|_{\partial\Omega}].
\end{equation}
Here $[f_\alpha]$ is the discretized version of $f_\alpha$ in~\eqref{eq:normal}; 
$\psi$ is an arbitrary harmonic function and $[\psi|_{\partial\Omega}] \in\mathbb{R}^{4I(L+1)\times 1}$ denotes the vectorized boundary restriction $\psi|_{\partial\Omega}$. Both $[f_\alpha]$ and $\psi|_{\partial\Omega}$ are in the lexicographical order as before.
Since $[K]$ has zero singular values, 
we solve~\eqref{eq:normal} with Tikhonov regularization. Specifically, the equation that we solve is
\begin{equation} \label{eq:regdiscretenormal}
([K]^t [K] + \alpha) [f_\alpha] = [K]^t [B] [\psi|_{\partial\Omega}]
\end{equation}
where $[K]^t$ is the transpose of $[K]$. 

% ...



\textbf{Solving for $[c^{-2}]$.}
The last step is to solve for $[c^{-2}]$. 
In the proof of Algorithm~\ref{alg:main}, this is accomplished by constructing appropriate complex exponential harmonic functions~\eqref{eq:phipsi} and inverting the Fourier transform. Nonetheless, such harmonic functions are not suitable for numerical implementation: they tend to blow up due to their exponential growth in certain directions. We instead exploit harmonic functions of the following form~\cite{martin2008polyhedral}
\begin{equation} \label{eq:harmonics}
    \sum^n_{j=1} a_j \Phi(|x-x^{(j)}|)
\end{equation}
Here $a_j$ are real scalars and 
$\Phi$ is (up to a constant factor) the fundamental solution of the Laplace operator: $\Phi(r) = \log r$ for $n=2$ and $\Phi(r) = \frac{1}{r}$ for $n\geq 3$. These functions are harmonic except at the singularities $x^{(j)}$.

We proceed to discretize the identity~\eqref{eq:limit}. On the right-hand side of~\eqref{eq:limit}, we fix a small $\alpha > 0$ and approximate the boundary integral over $[0,T]\times\partial\Omega$ using the trapezoidal rule:
\begin{equation} \label{eq:discretelimitleft}
 (f_\alpha, B\phi)_{L^2((0,T)\times\partial\Omega)} \approx \sum^{4I(L+1)}_{j=1} w_j [f_\alpha]_j [B\phi|_{\partial\Omega}]_j
\end{equation}
where $[f_\alpha]$ has been obtained from the previous step, and $[B\phi|_{\partial\Omega}]$ is computed from the matrix multiplication $[B\phi|_{\partial\Omega}] = [B] [\phi|_{\partial\Omega}]$. On the left-hand side of~\eqref{eq:limit}, we approximate the interior integral over $\Omega$ by successively applying the trapezoidal rule first to $y$ and then to $x$. If we write $w = (\frac{1}{2}, 1, \dots, 1, \frac{1}{2})\in \mathbb{R}^{I+1}$ for the coefficient vector of the trapezoidal rule, 
then
\begin{align}
(\psi,\phi)_{L^2(\Omega,c^{-2}dx)} & = \int^{1}_{-1} \int^{1}_{-1} \psi(x,y) \phi(x,y) c^{-2}(x,y) \, dx dy \nonumber \\
 & \approx 
\sum^I_{j,k=0} w_j w_k \psi(x_j,y_k) \phi(x_j,y_k) c^{-2}(x_j,y_k) ( \Delta x )^2. \label{eq:discretelimitright}.
\end{align}



Finally, we equating~\eqref{eq:discretelimitleft} and~\eqref{eq:discretelimitright} and inserting various harmonic functions of the form~\eqref{eq:harmonics}. This gives rise to a system of linear equations on the unknowns $c^{-2}(x_j,y_k)$, $j,k=0,1,\dots,I$. If the number of harmonic functions is small, the linear system will be under-determined. In this circumstance, we employ Tikhonov regularization to solve for the regularized unknowns.

\section{Numerical Experiments} \label{sec:examples}

We validate the reconstruction algorithm in this section with several numerical examples. All the numerical experiments are conducted on a Windows 10 laptop with Intel Core i7-9750H 2.6GHz CPU and 16GB RAM.

For the forward simulation, we employ a computational grid of size $(2L+1)\times101\times101$ to generate the ND map. For the inverse problem, we re-sample the ND map on a coarser grid of size $(L+1)\times51\times51$ and implement Algorithm~\ref{alg:main} there to avoid the inverse crime. Here the value of $L$ depends on the choice of $c$.

We construct the following harmonic functions in view of~\eqref{eq:harmonics}:
\begin{align*}
    \phi^{(1)} = \ln((x-2.3)^2+(y-2.2)^2),
& \quad\quad \phi^{(2)} = \ln((x+2.5)^2+(y-2.1)^2),\\
\phi^{(3)} = \ln((x-2.7)^2+(y+1.9)^2),
& \quad\quad \phi^{(4)} = \ln((x+1.5)^2+(y+2.5)^2),\\
\phi^{(5)} = \ln((x+1.2)^2+(y+2.5)^2), & \quad\quad \phi^{(6)} = 1.
\end{align*}
We denote the vector space generated by the products of these harmonic functions by $S_6$, that is,
$$
S_6 := \vspan\{\phi^{(i)}\phi^{(j)}: i,j = 1, \dots, 6\}.
$$








\bigskip \bigskip
\textbf{Experiment 1: $c \equiv 1$ and $c^{-2} \in S_6$.}

We test the reconstruction of a constant speed $c\equiv 1$ in this experiment, see Figure~\ref{fig:exp1_c} for the ground-truth speed. 
Notice that $c^{-2}\equiv 1 \in S_6$ since
$\phi^{(6)}=1$.
The reconstructed images along with the errors are illustrated in Figure~\ref{fig:exp1}, in the presence of $0\%$, $5\%$ and $50\%$ of Gaussian random noises with zero mean and unit variance respectively. We observe that the addition of the random noise has almost negligible impact on the reconstructed images. This is because in the definition~\eqref{eq:K} of $K$, the ND map is followed by the low-pass filter $J$, which tends to smoothing out the random noise. We plot the image of $[K]$ before (Figure~\ref{fig:exp1_c}) and after (Figure~\ref{fig:Knoise1}) adding the noise. 
As a justification, we also test the impact of non-random noise on the reconstruction. We re-run the code with constant noise $0.01$, $0.02$, and $0.05$ added to the ND map. In this case, the filter fails to smooth out the noise, and the reconstructions are significantly compromised, see Figure~\ref{fig:Knoise0}

\begin{figure}[!htb]
    \centering
    \includegraphics[width=0.48\textwidth]{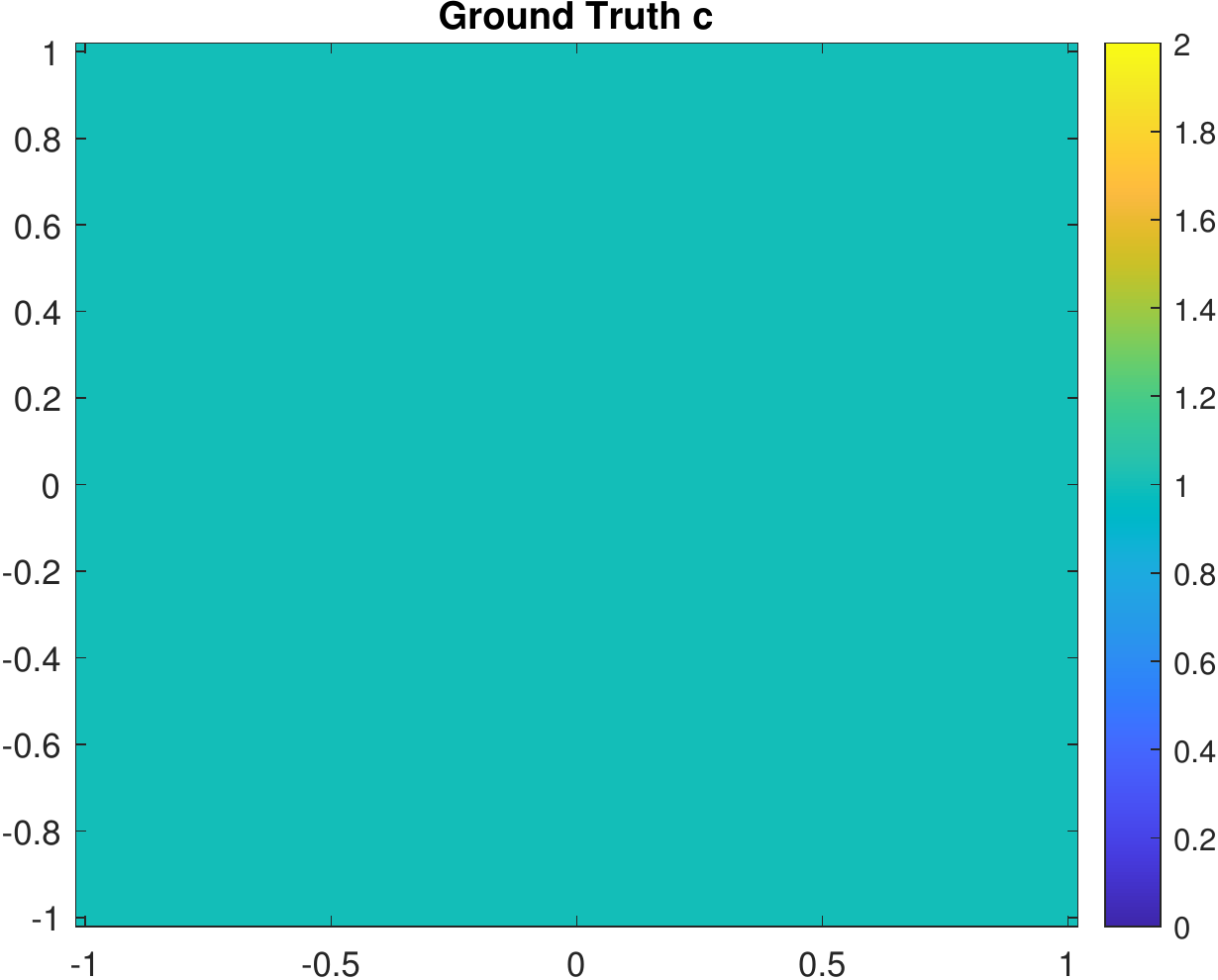}
    \includegraphics[width=0.5\textwidth]{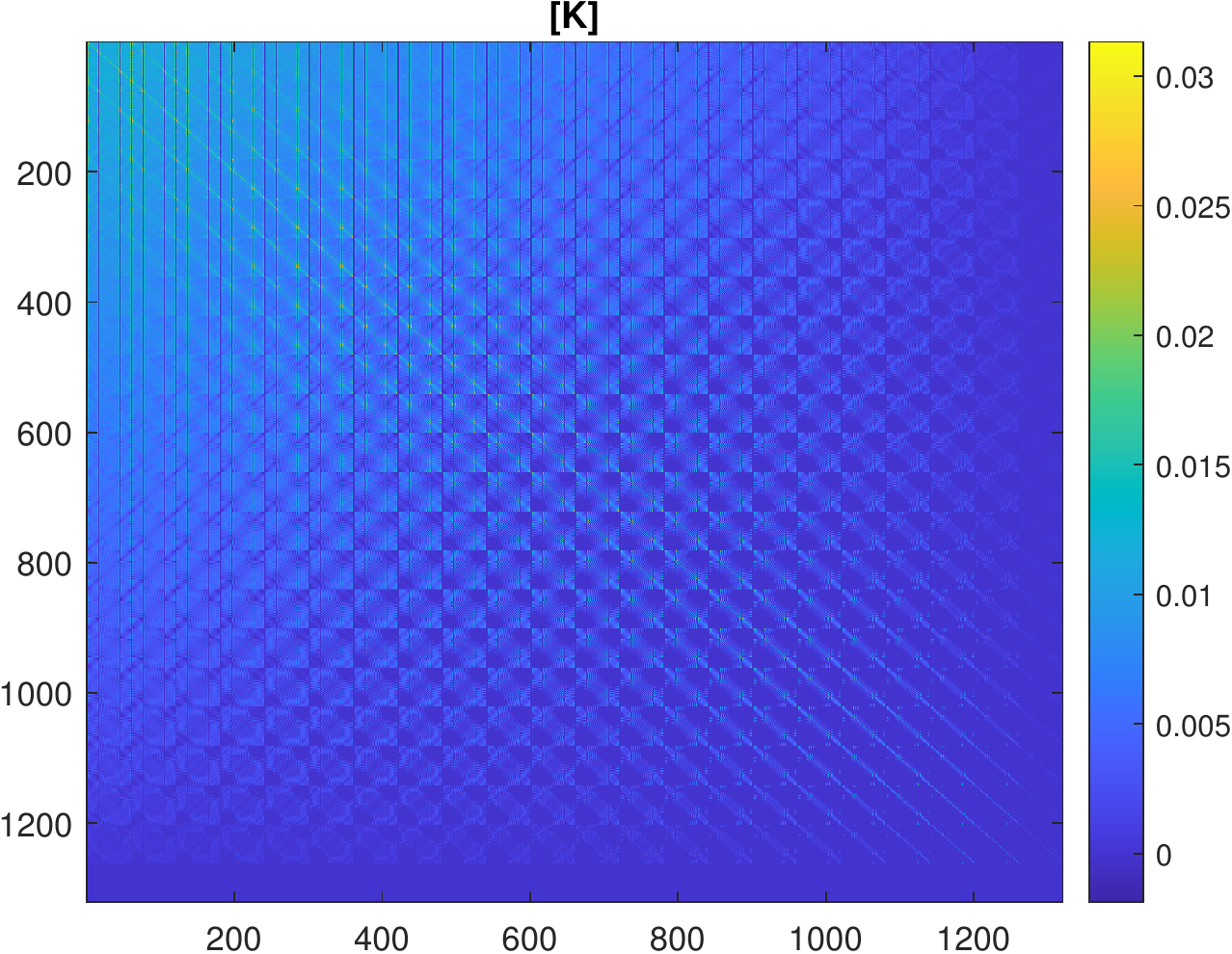}
    \caption{Left: The constant speed $c\equiv 1$. Grid: $283\times51\times51$, $I=50,L=282$. Right: $[K]$. Grid: $43\times16\times16$, $I=15,L=42$.}
    \label{fig:exp1_c}
\end{figure}
\begin{figure}[p]
    \centering
    \includegraphics[width=0.49\textwidth]{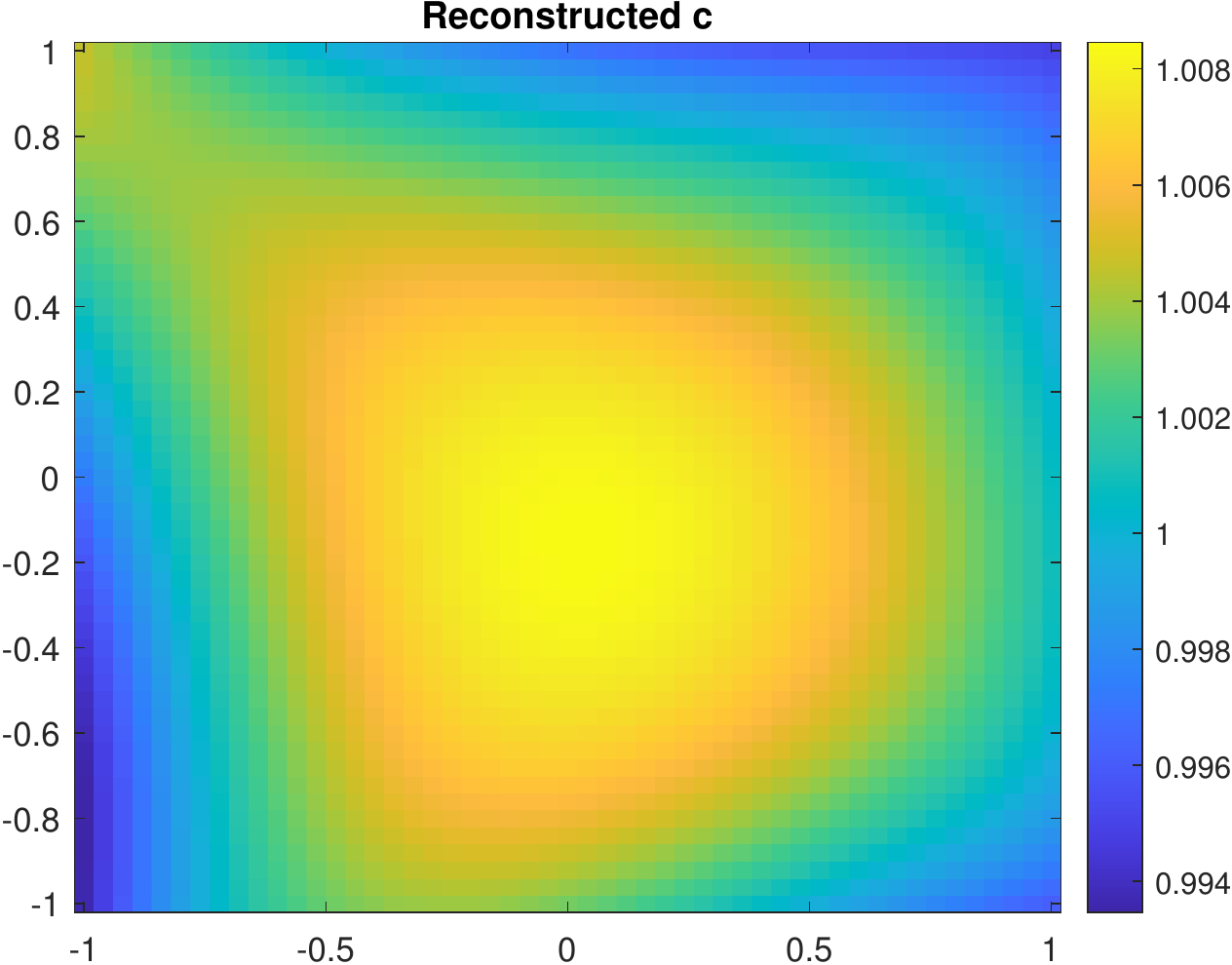}
    \includegraphics[width=0.49\textwidth]{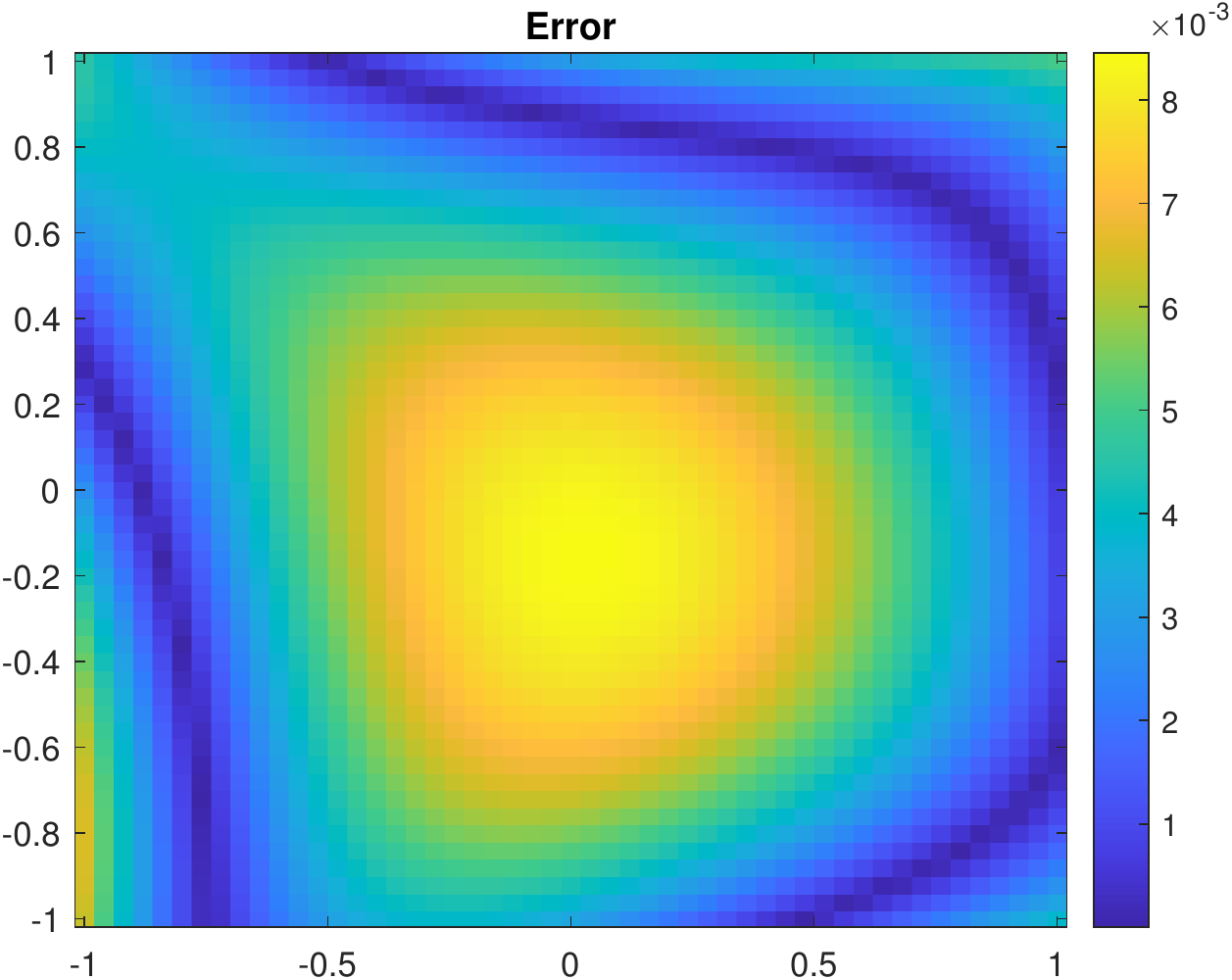}

    \includegraphics[width=0.49\textwidth]{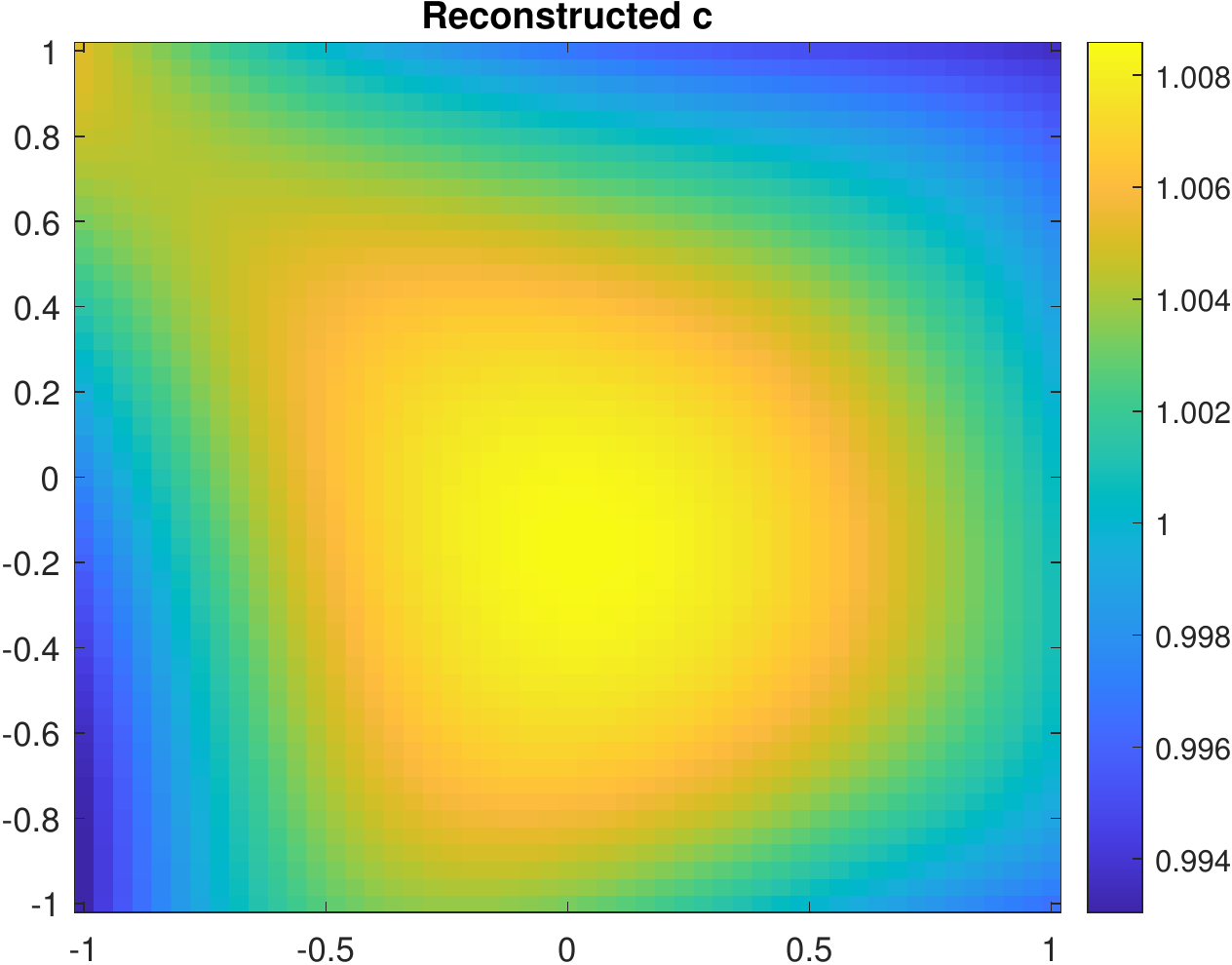}
    \includegraphics[width=0.49\textwidth]{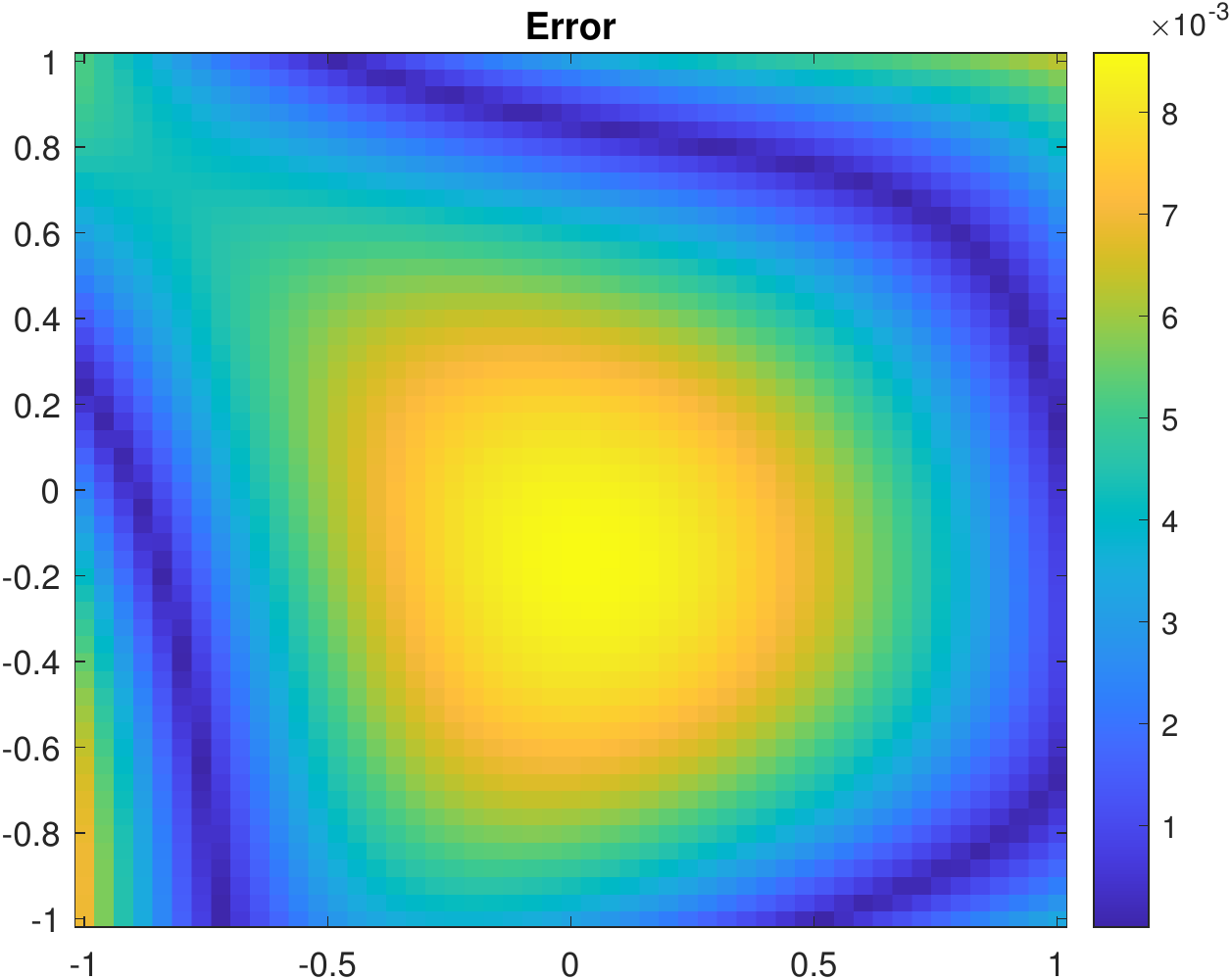}
    
    \includegraphics[width=0.49\textwidth]{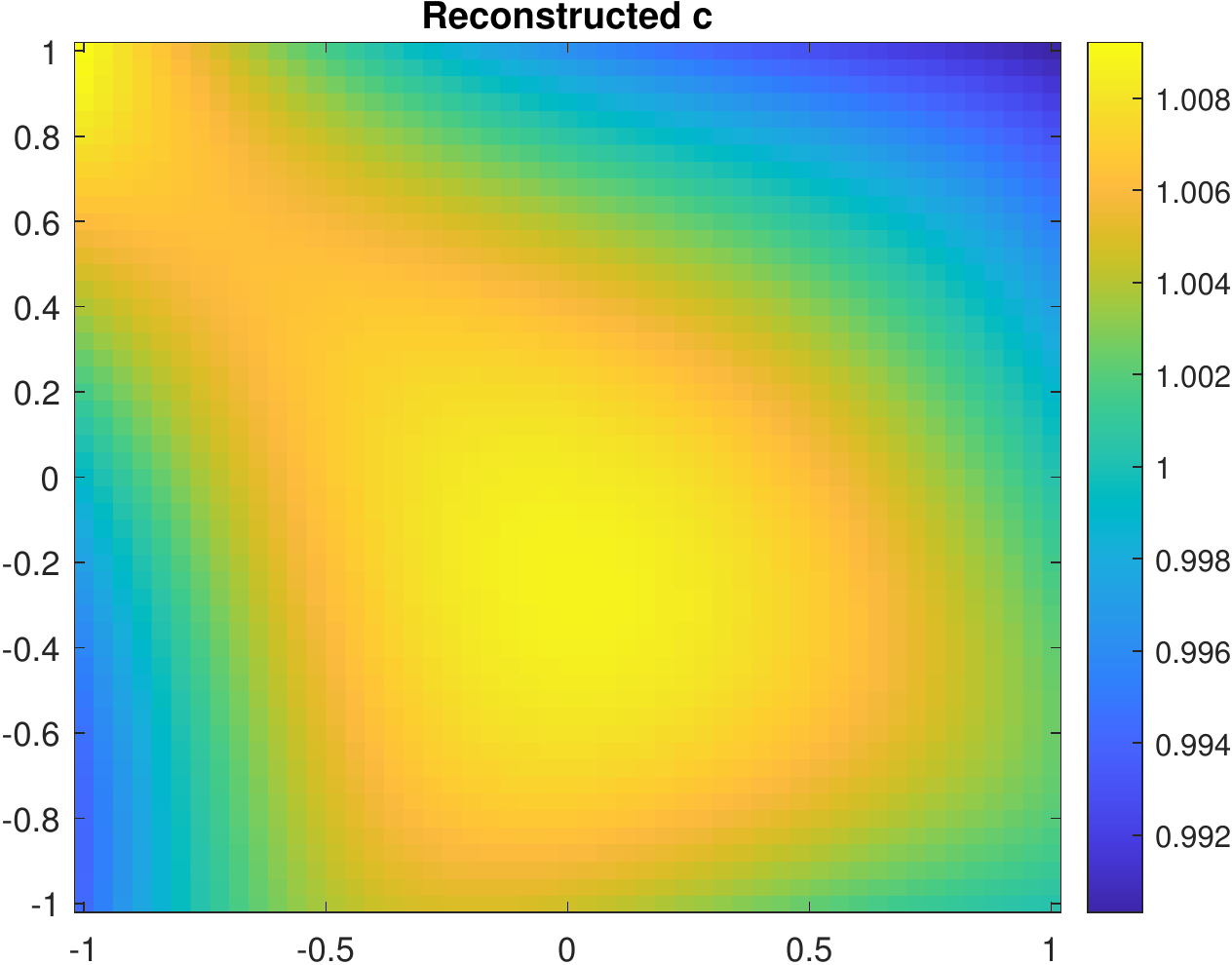}
    \includegraphics[width=0.49\textwidth]{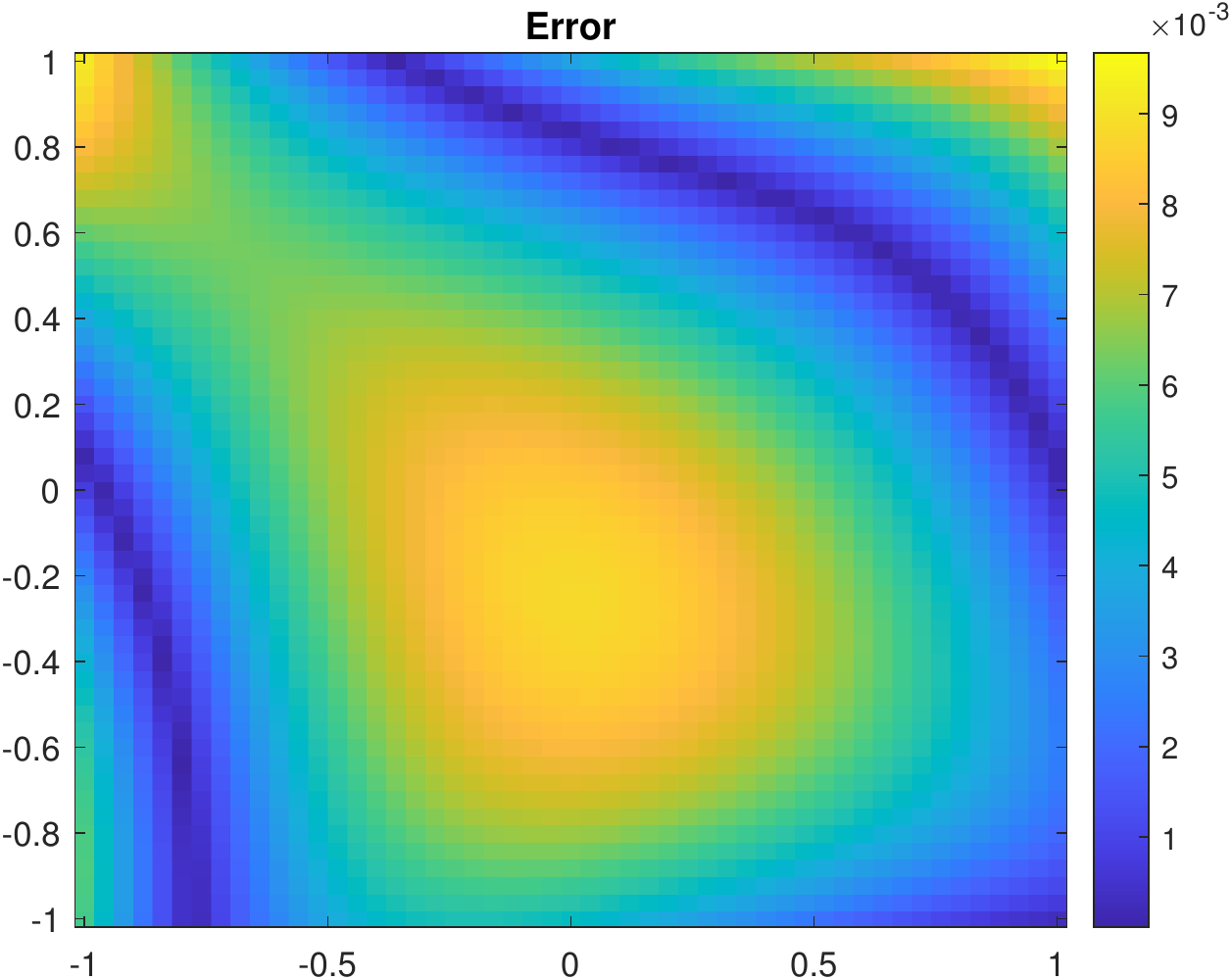}
    \caption{Reconstructions of the constant speed $c=1$.
    Left column: reconstructed $c$. Right column: error between the reconstruction and the ground truth. First row: $0\%$ noise; the relative $L^2$-error is $0.4769\%$.
    Second row: $5\%$ noise; the relative $L^2$-error is $0.4873\%$.
    Third row: $50\%$ noise; the relative $L^2$-error is $0.5454\%$.
    Grid: $283\times51\times51$, $I=50,L=282$.
    }
    \label{fig:exp1}
\end{figure}

\begin{figure}[p]
    \centering
    \includegraphics[width=0.49\textwidth]{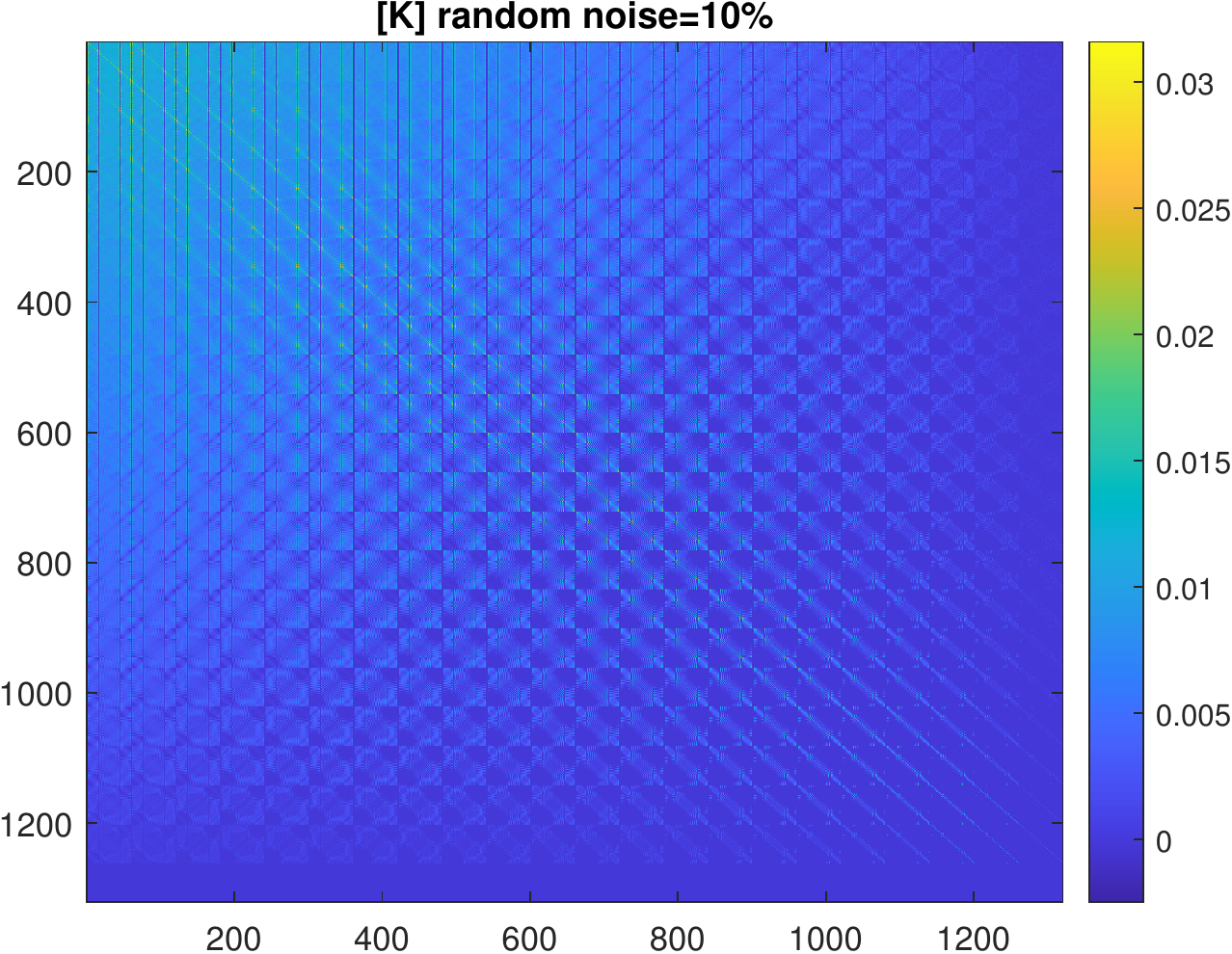}
    \includegraphics[width=0.49\textwidth]{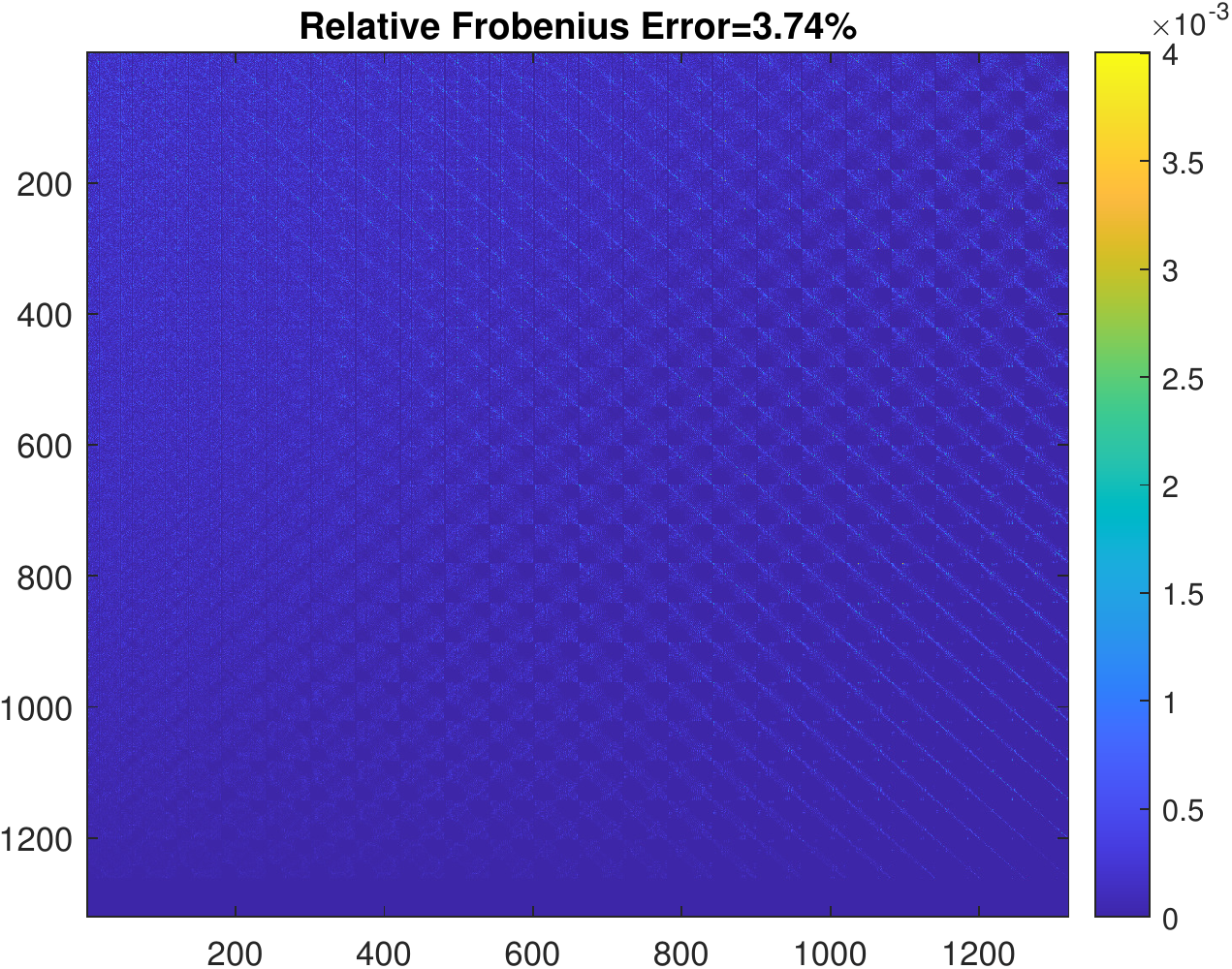}
    
    \includegraphics[width=0.49\textwidth]{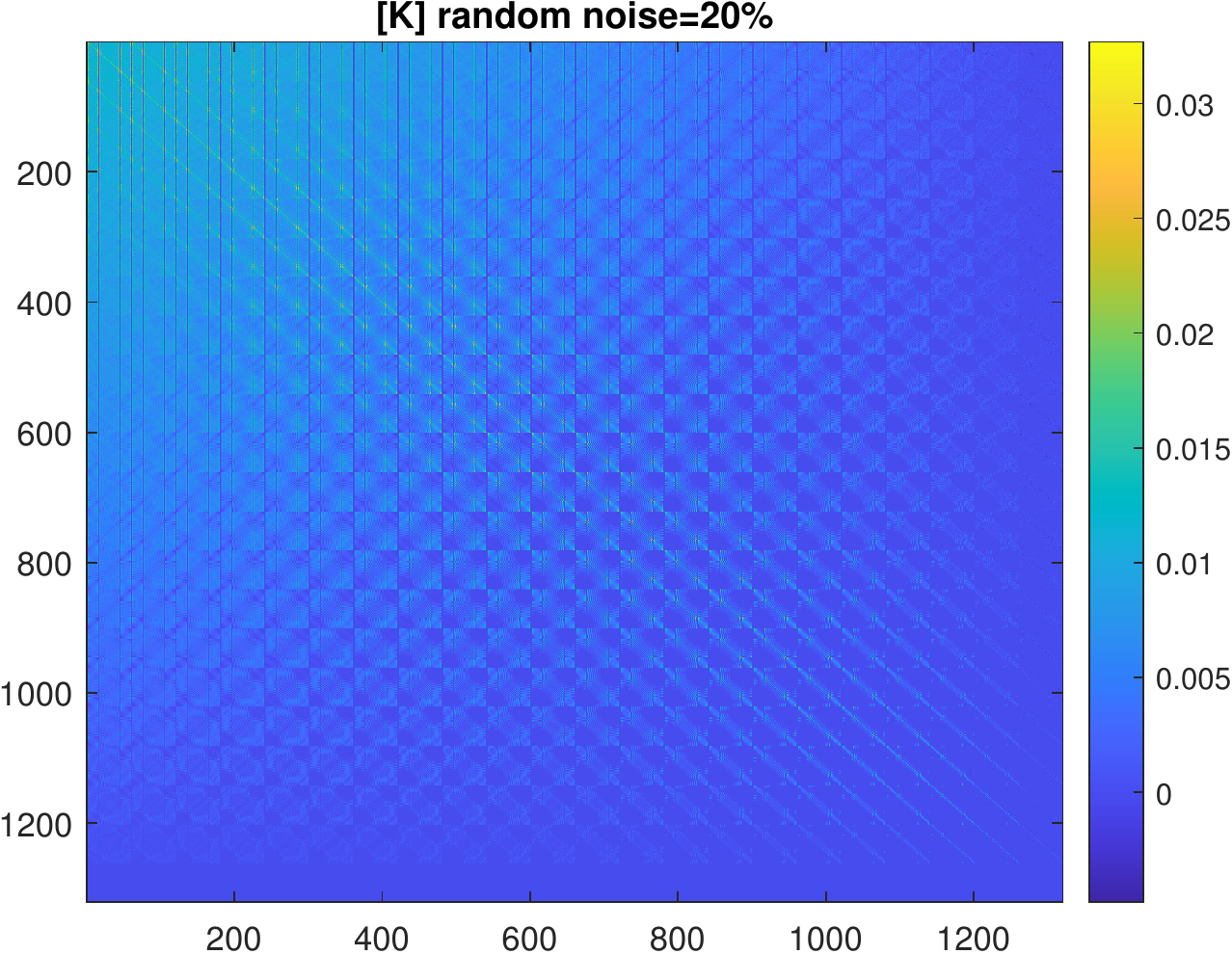}
    \includegraphics[width=0.49\textwidth]{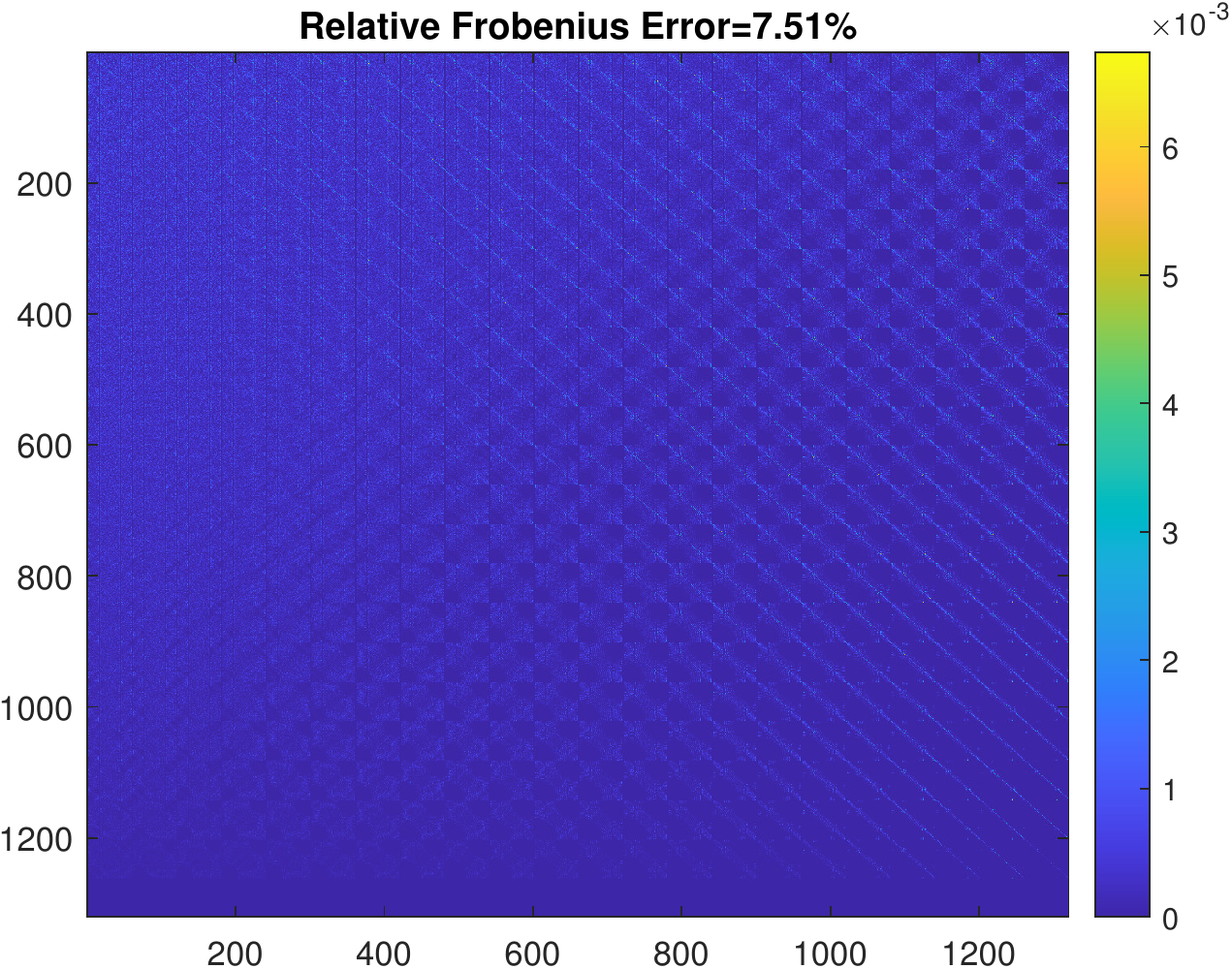}
    
    \includegraphics[width=0.49\textwidth]{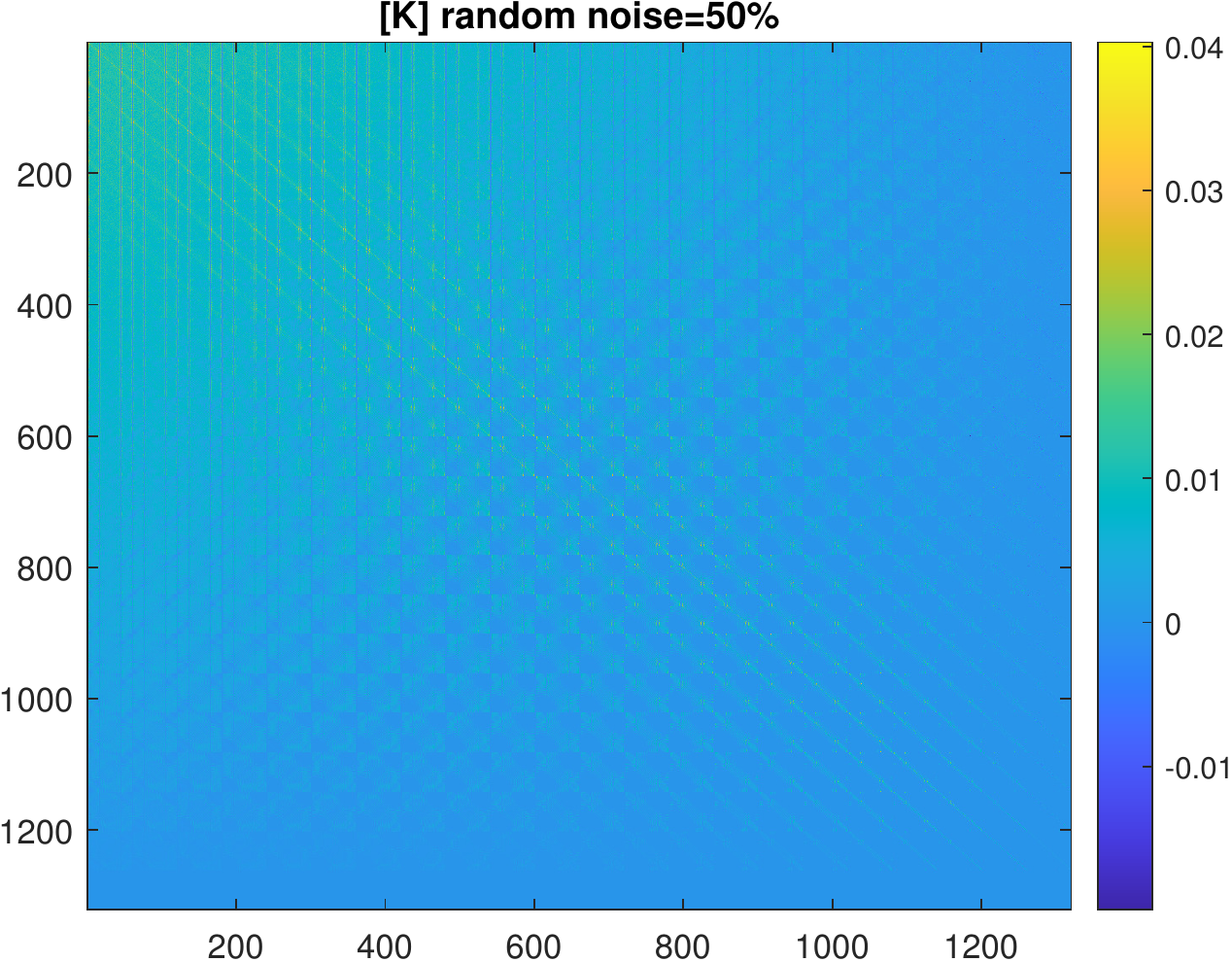}
    \includegraphics[width=0.49\textwidth]{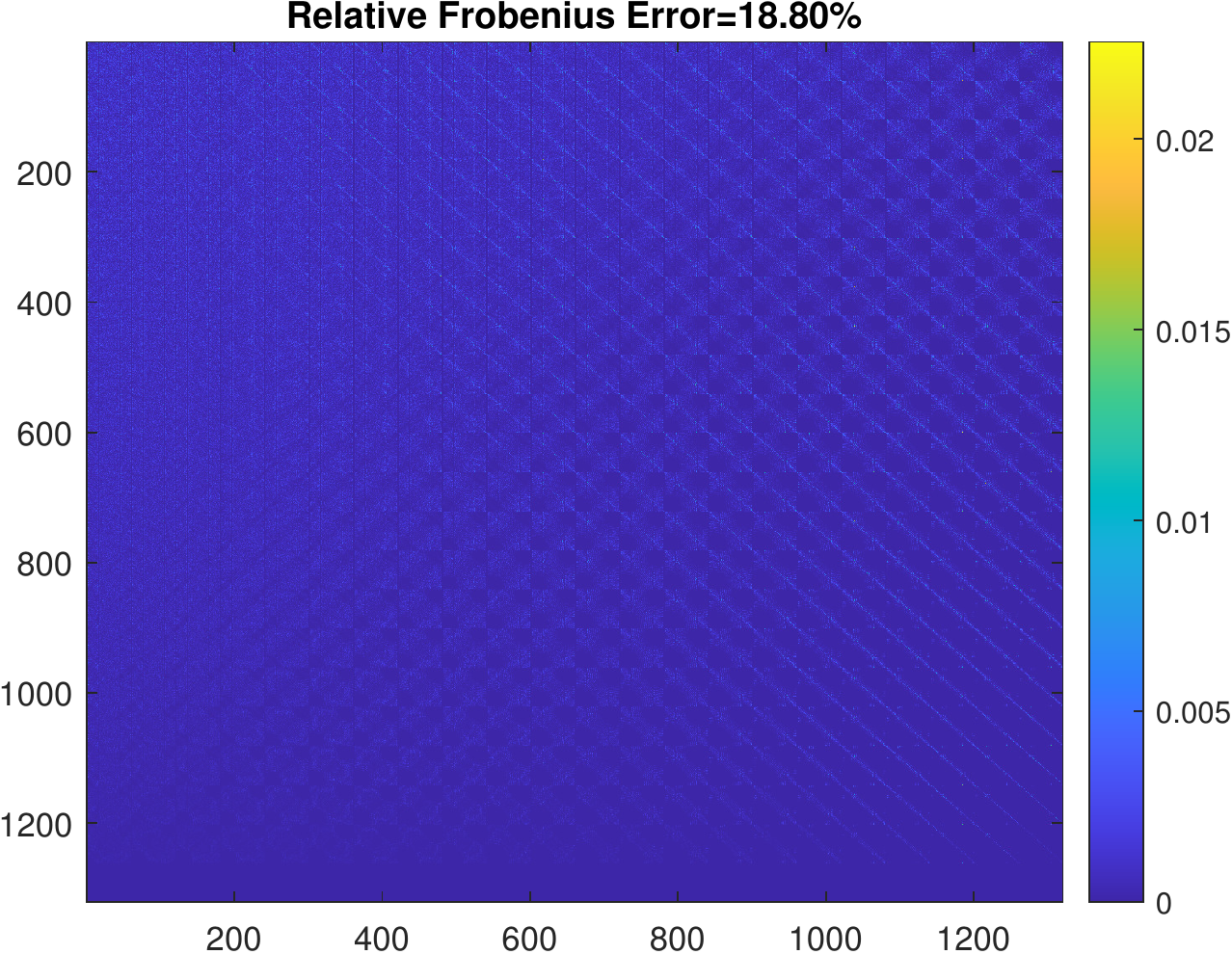}
    \caption{
    Left: $[K]$ with random noise.
    Right: absolute difference between $[K]$ with random noise and $[K]$ with no noise.
    First row: $10\%$ of random noise.
    Second row: $20\%$ of random noise.
    Third row: $50\%$ of random noise.
    Grid: $43\times16\times16$, $I=15,L=42$.
}
    \label{fig:Knoise1}
\end{figure}

\begin{figure}[p]
    \centering
    \includegraphics[width=0.49\textwidth]{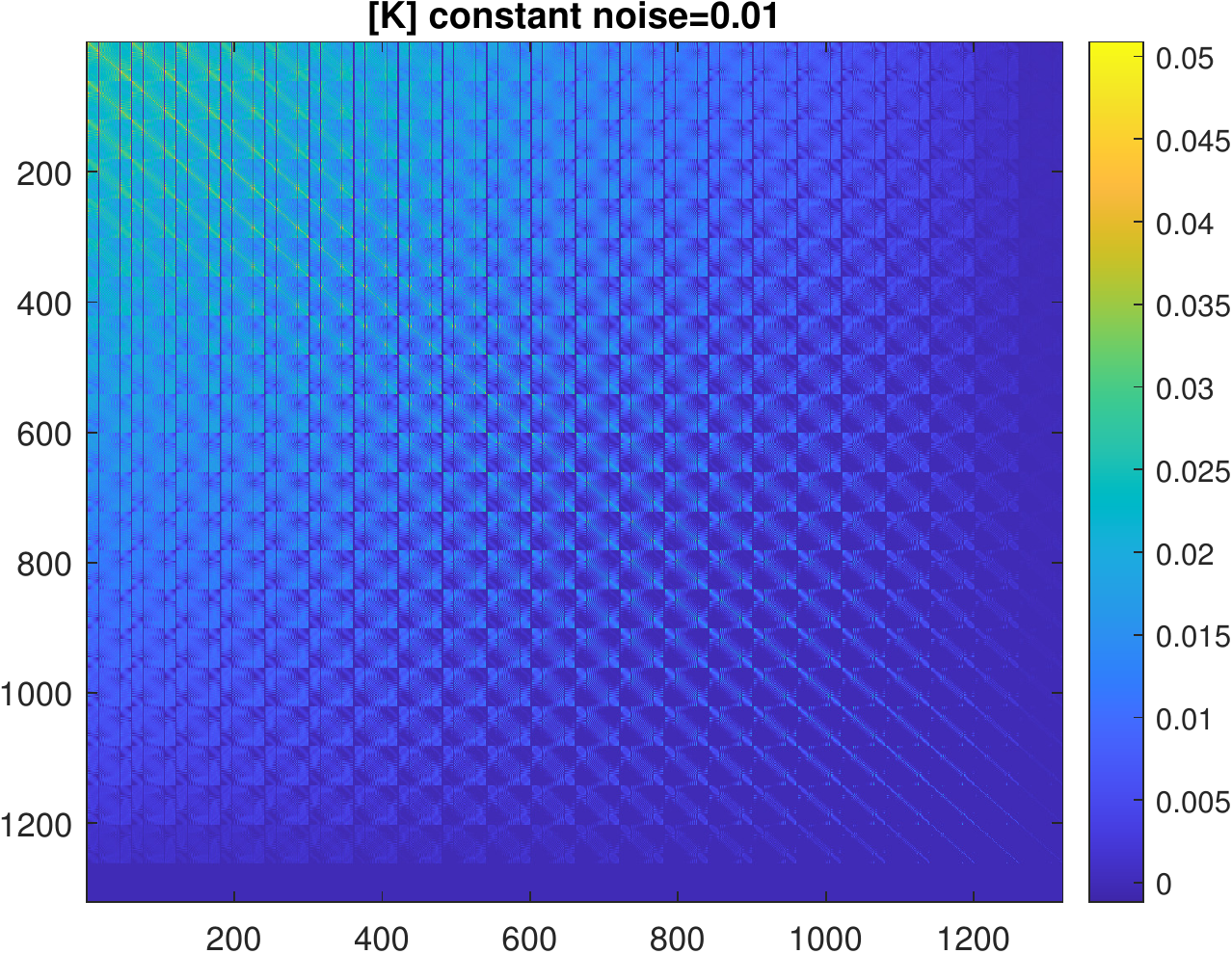}
    \includegraphics[width=0.49\textwidth]{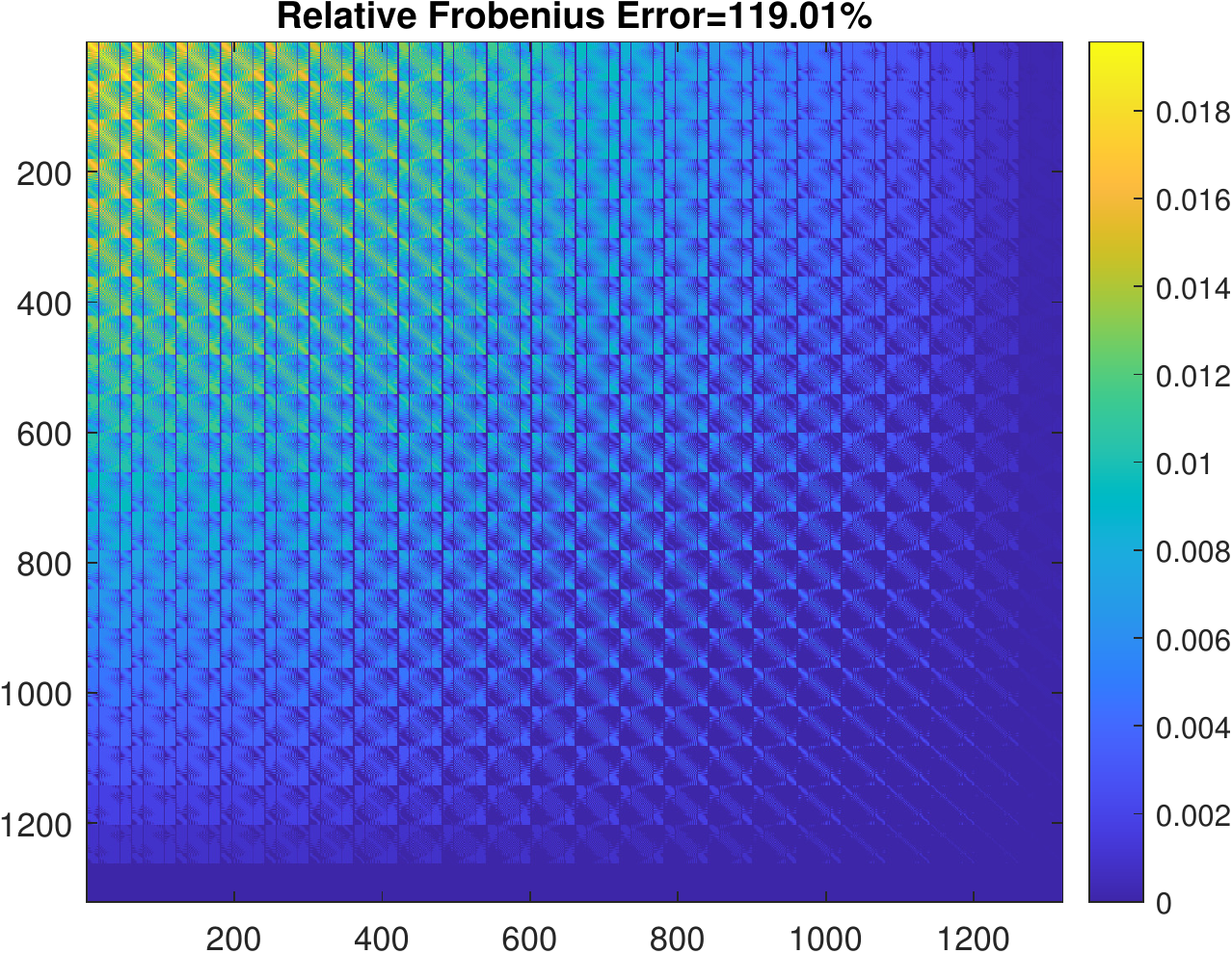}
    
    \includegraphics[width=0.49\textwidth]{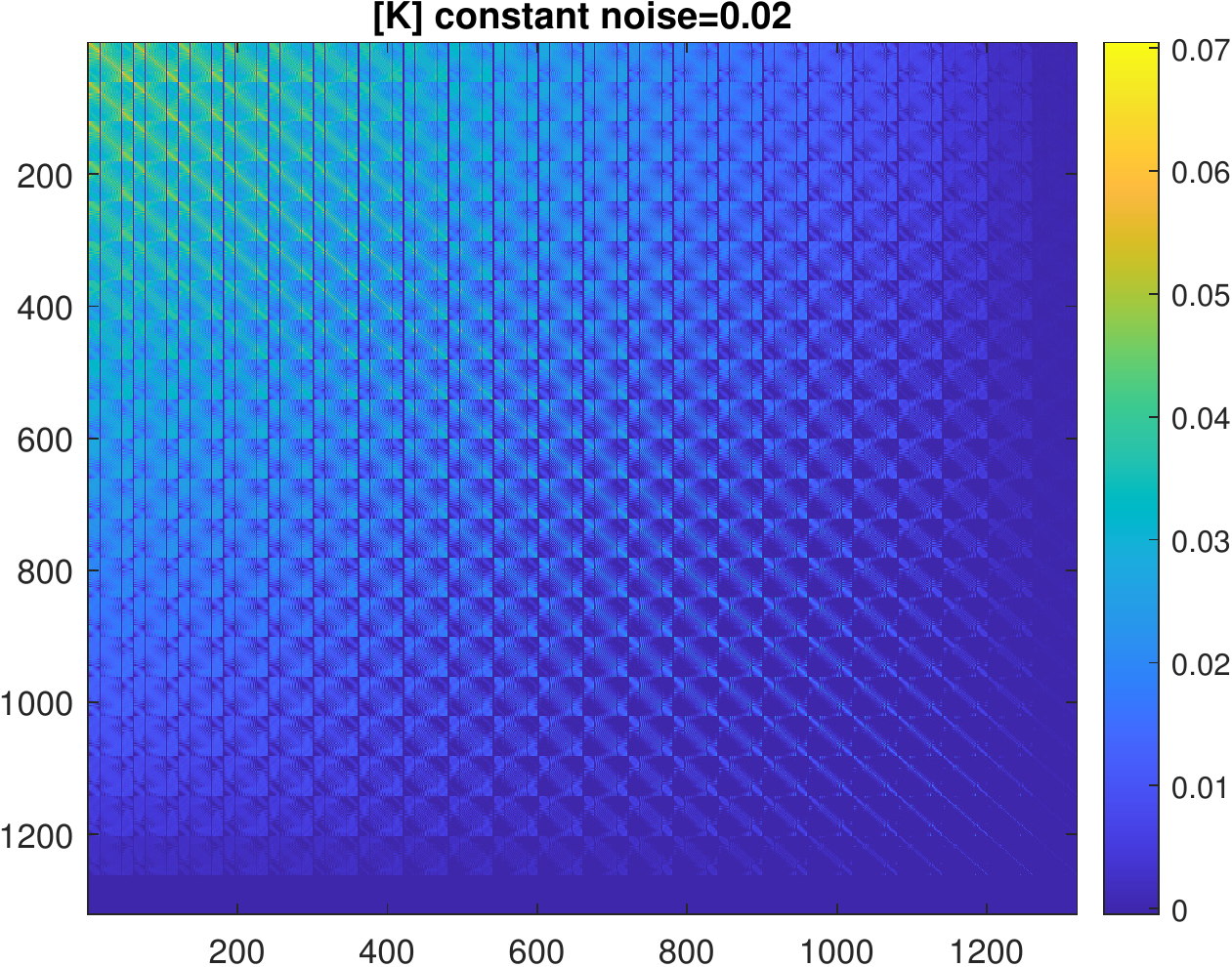}
    \includegraphics[width=0.49\textwidth]{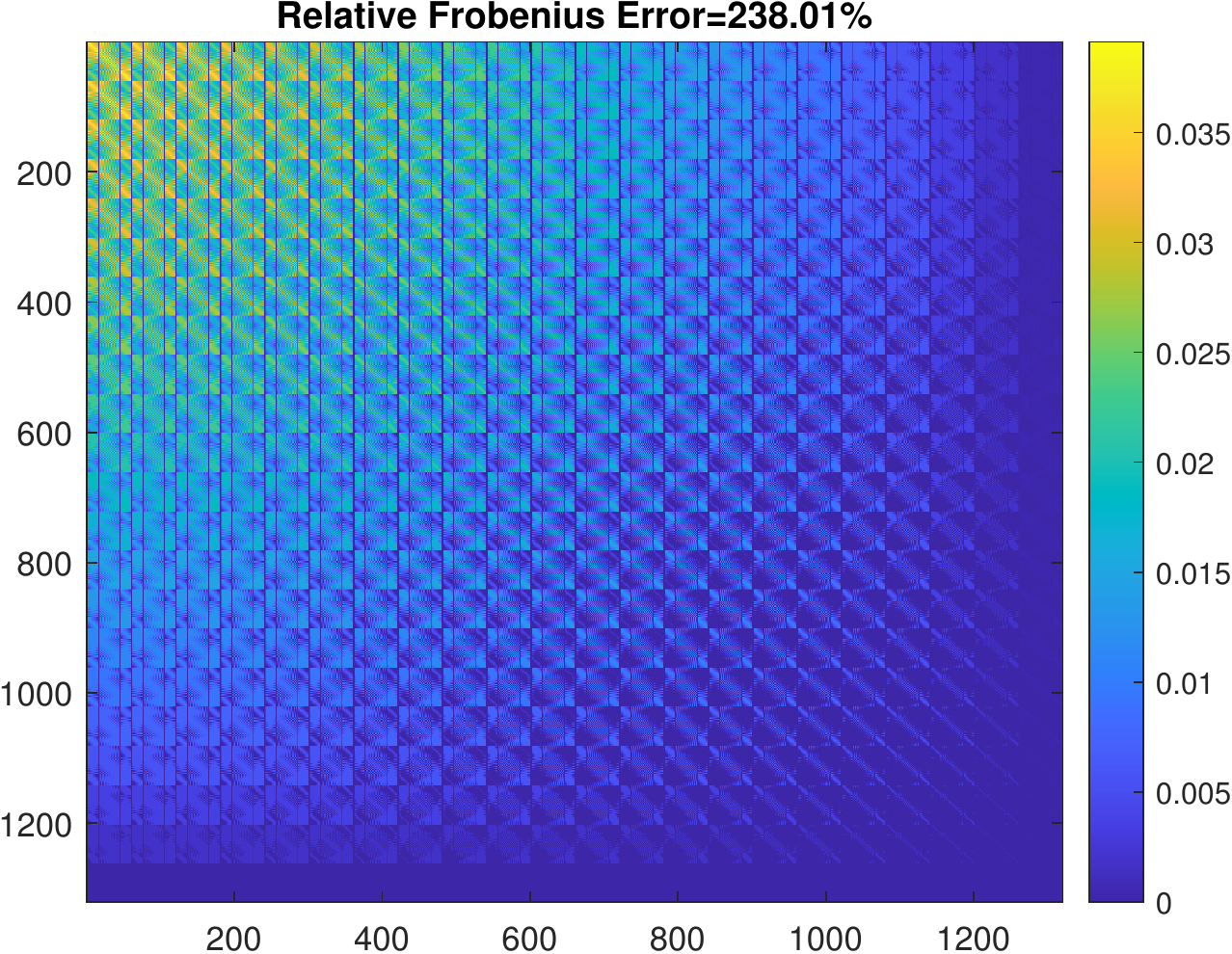}
    
    \includegraphics[width=0.49\textwidth]{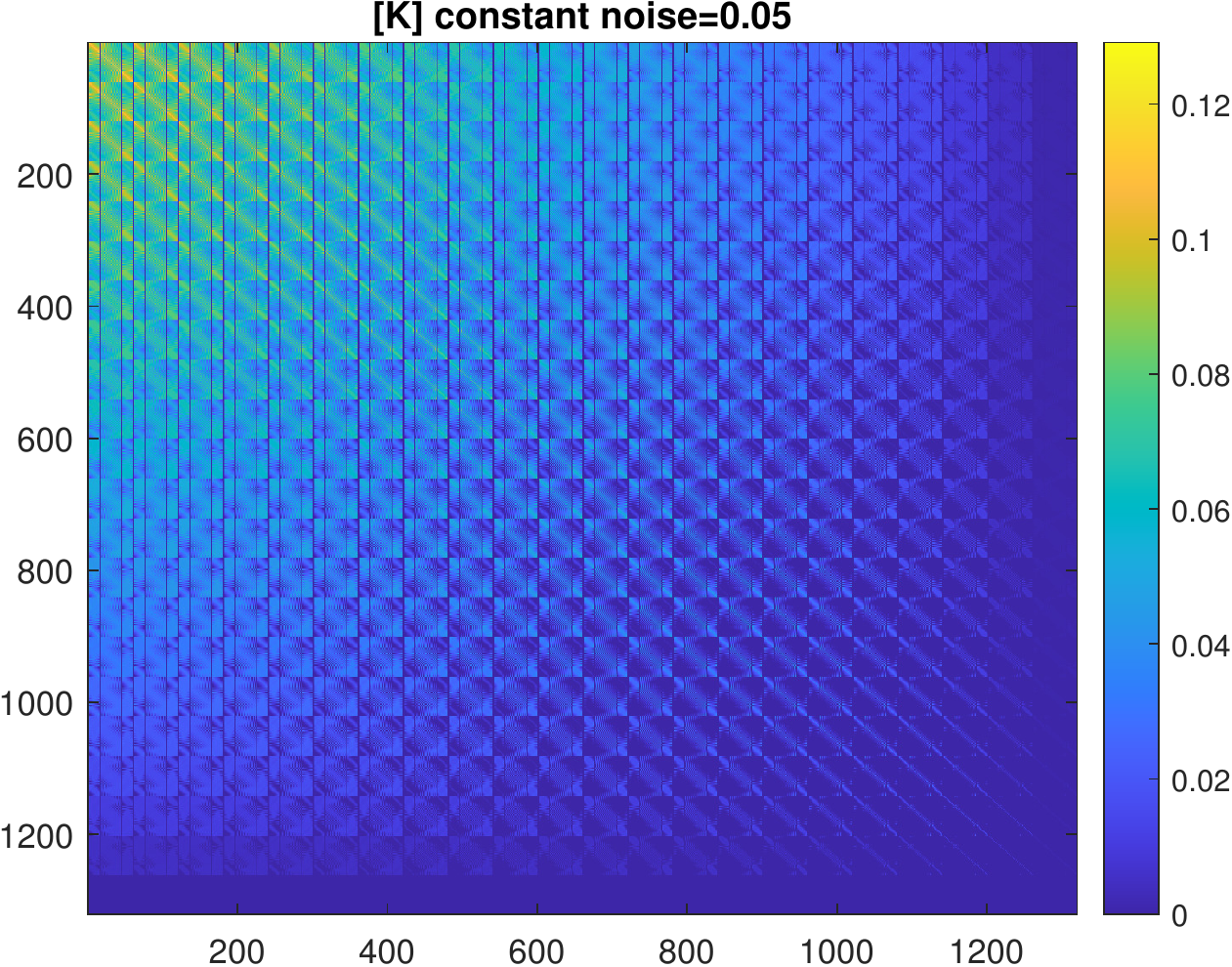}
    \includegraphics[width=0.49\textwidth]{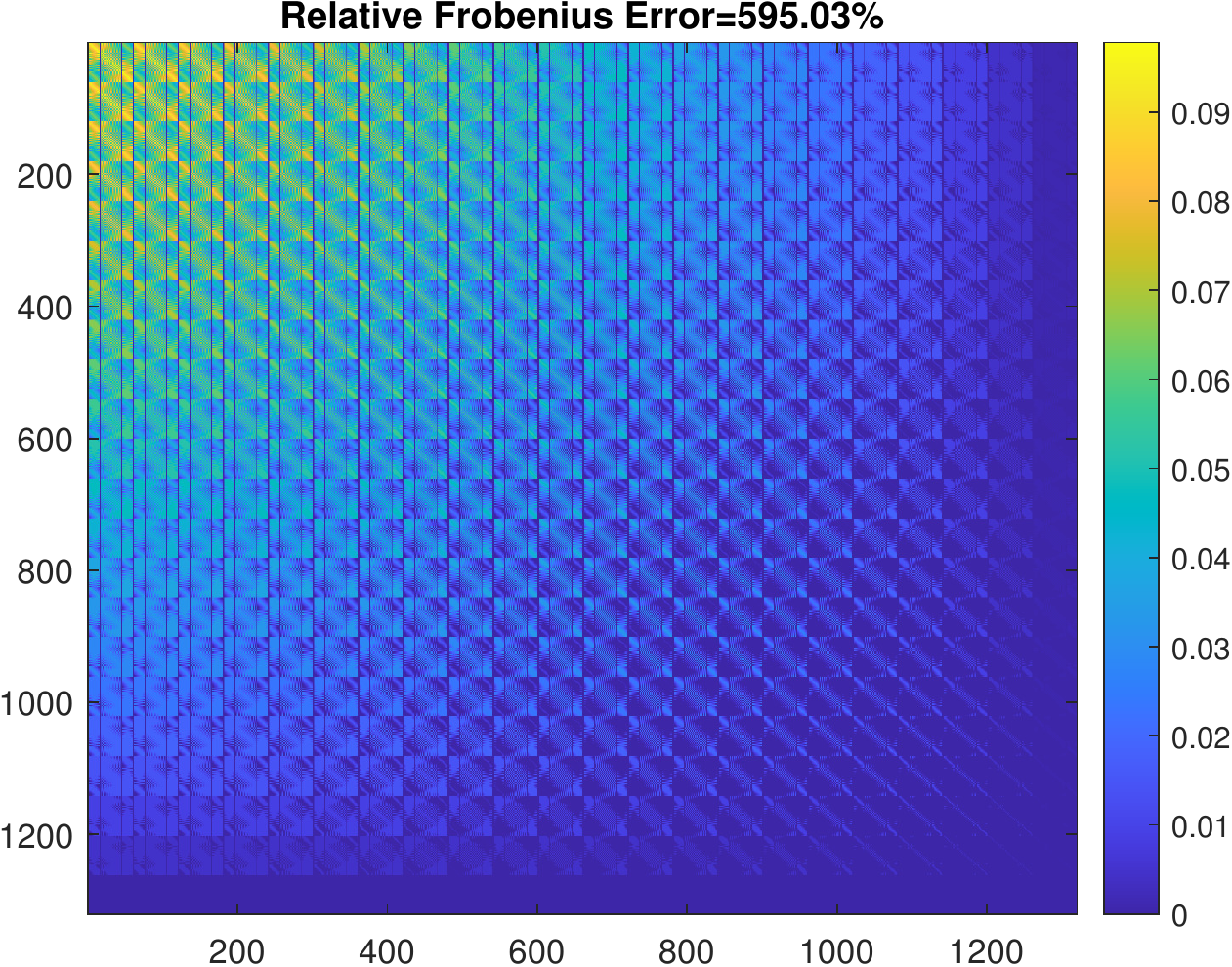}
    \caption{
    Left: $[K]$ with constant noise.
    Right: absolute difference between $[K]$ with constant noise and $[K]$ with no noise.
    First row: constant noise = 0.01.
    Second row: constant noise = 0.02.
    Third row: constant noise = 0.05.
    Grid: $43\times16\times16$, $I=15,L=42$.
    }
    \label{fig:Knoise0}
\end{figure}

\bigskip \bigskip
\textbf{Experiment 2: $c$ is variable and $c^{-2} \notin S_6$.}

Next, we test the ability of the algorithm in recovering a variable speed $c$ with $c^{-2} \notin S_6$. The speed in use is
$$
    c(x,y) = 1+0.08\sin{\pi x}+0.06\cos{\pi y},
$$
as is illustrated in Figure~\ref{fig:exp2_c}.
The reconstructed images of $c$ with $0\%$, $5\%$ and $50\%$ of noise are shown in Figure~\ref{fig:exp2}. 
In this case, we cannot expect to reconstruct the exact discrete version of $c^{-2}$. Instead, what the algorithm yields is the $L^2$-orthogonal projection of $c^{-2}$ onto the subspace $S_6$. This is due to the use of Tikhonov regularization when solving for $[c^{-2}]$. See the numerical validation in Figure~\ref{fig:exp2}.


\begin{figure}[!htb]
    \centering
    \includegraphics[width=0.49\textwidth]{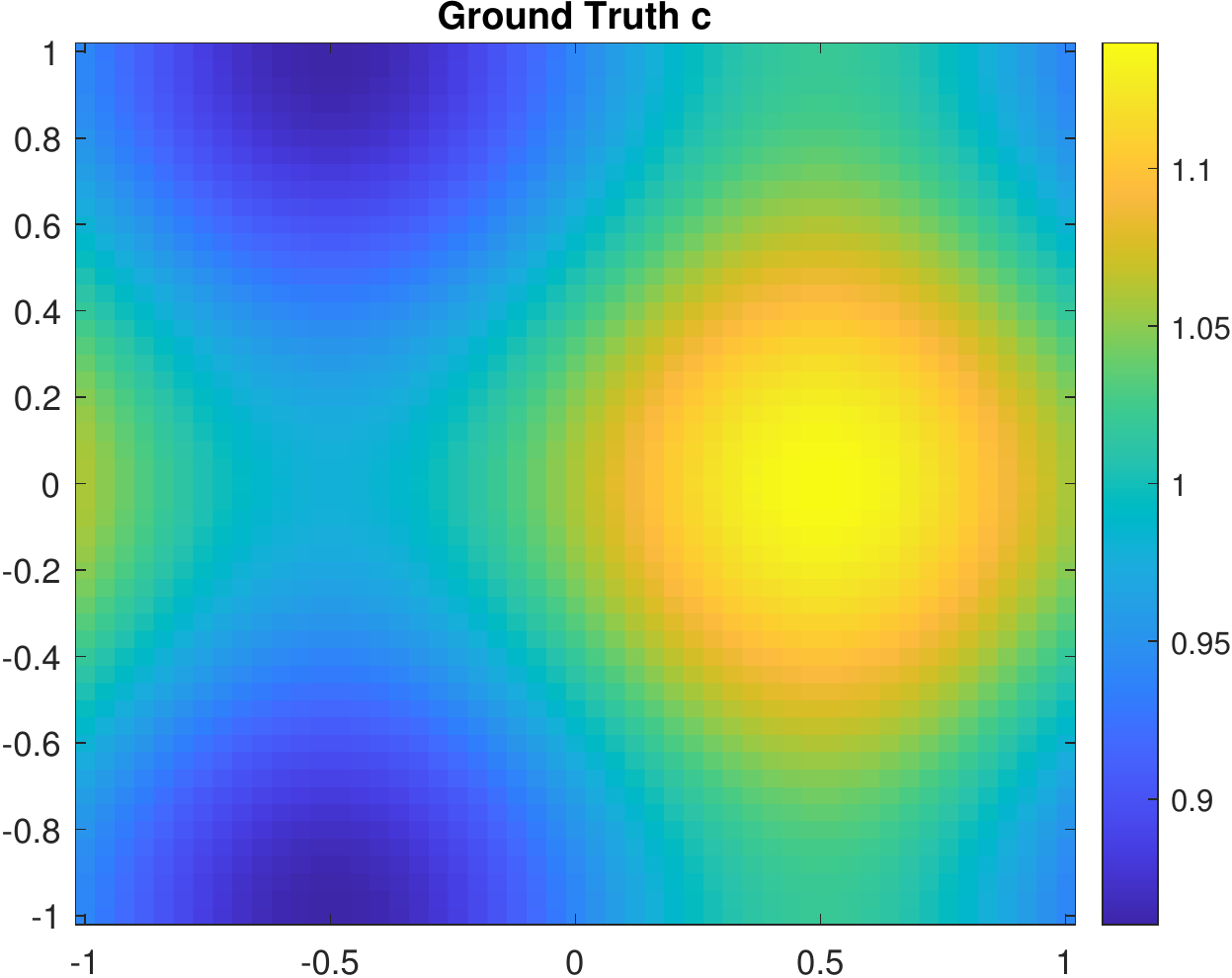}
    \includegraphics[width=0.49\textwidth]{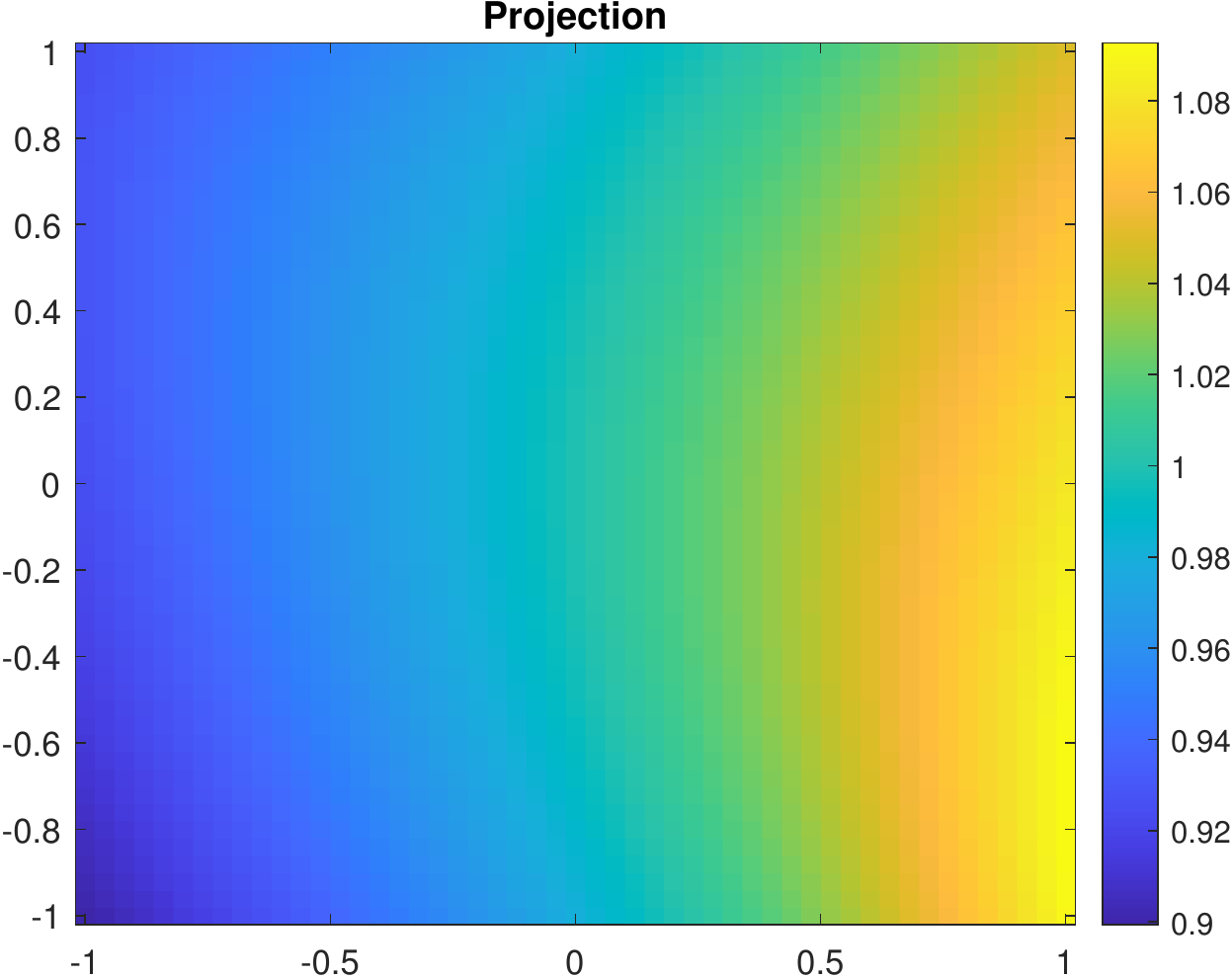}
    \caption{
     Left: the variable speed $c(x,y) = 1+0.08\sin{\pi x}+0.06\cos{\pi y}$. 
     Right: orthogonal projection of $c$ on $S_6$.}
    \label{fig:exp2_c}
\end{figure}
\begin{figure}[p]
    \centering
    \includegraphics[width=0.49\textwidth]{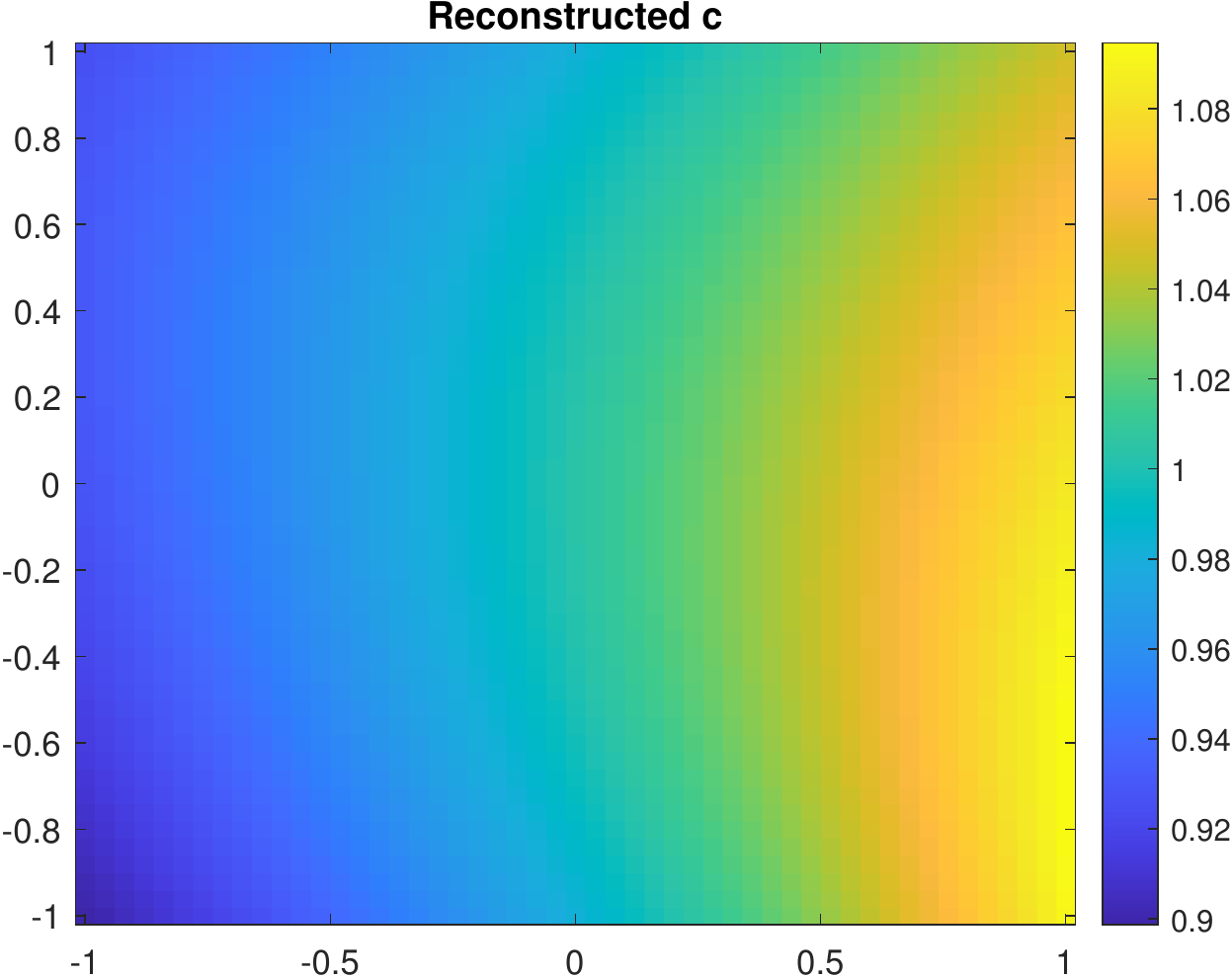}
    \includegraphics[width=0.49\textwidth]{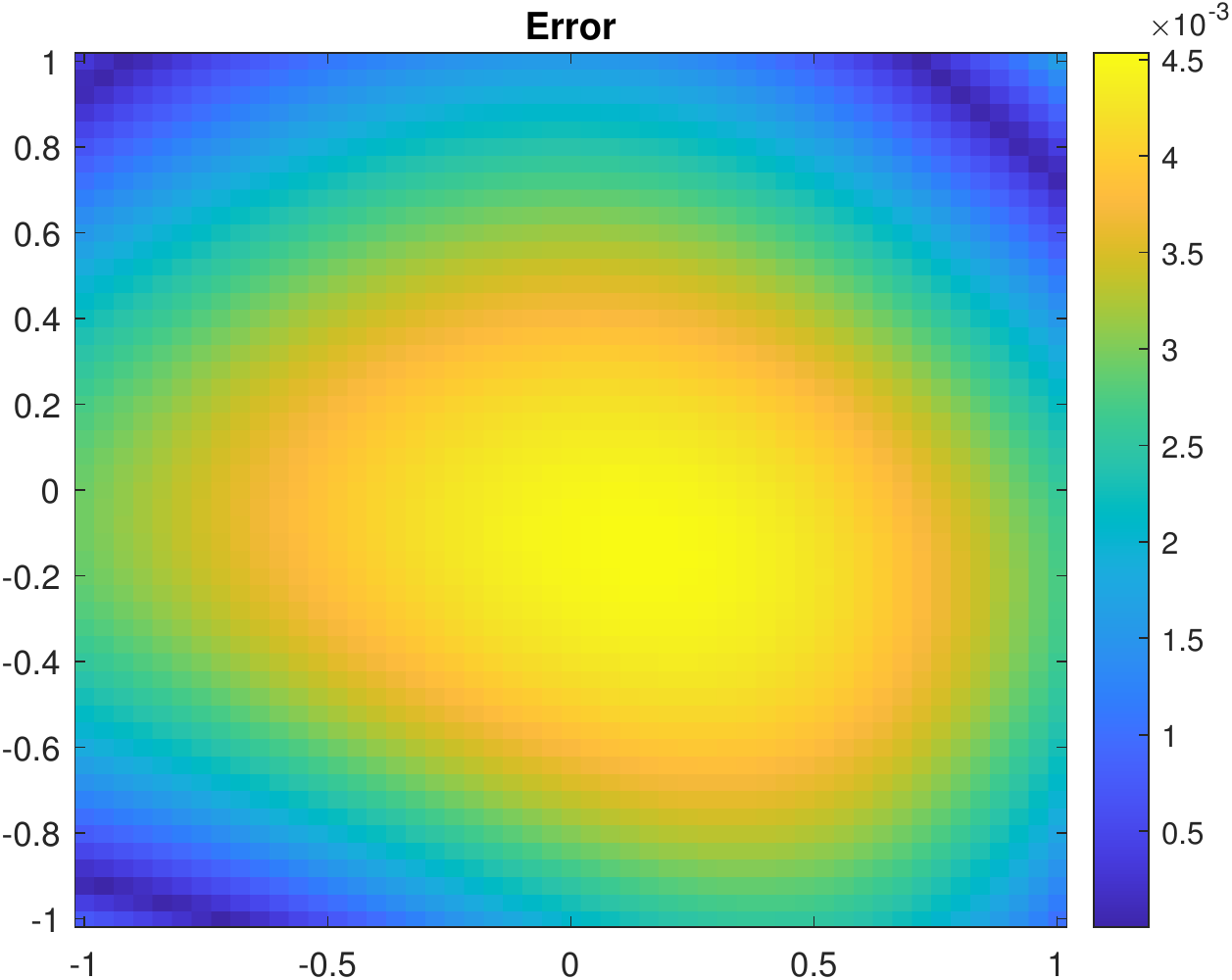}
    
    \includegraphics[width=0.49\textwidth]{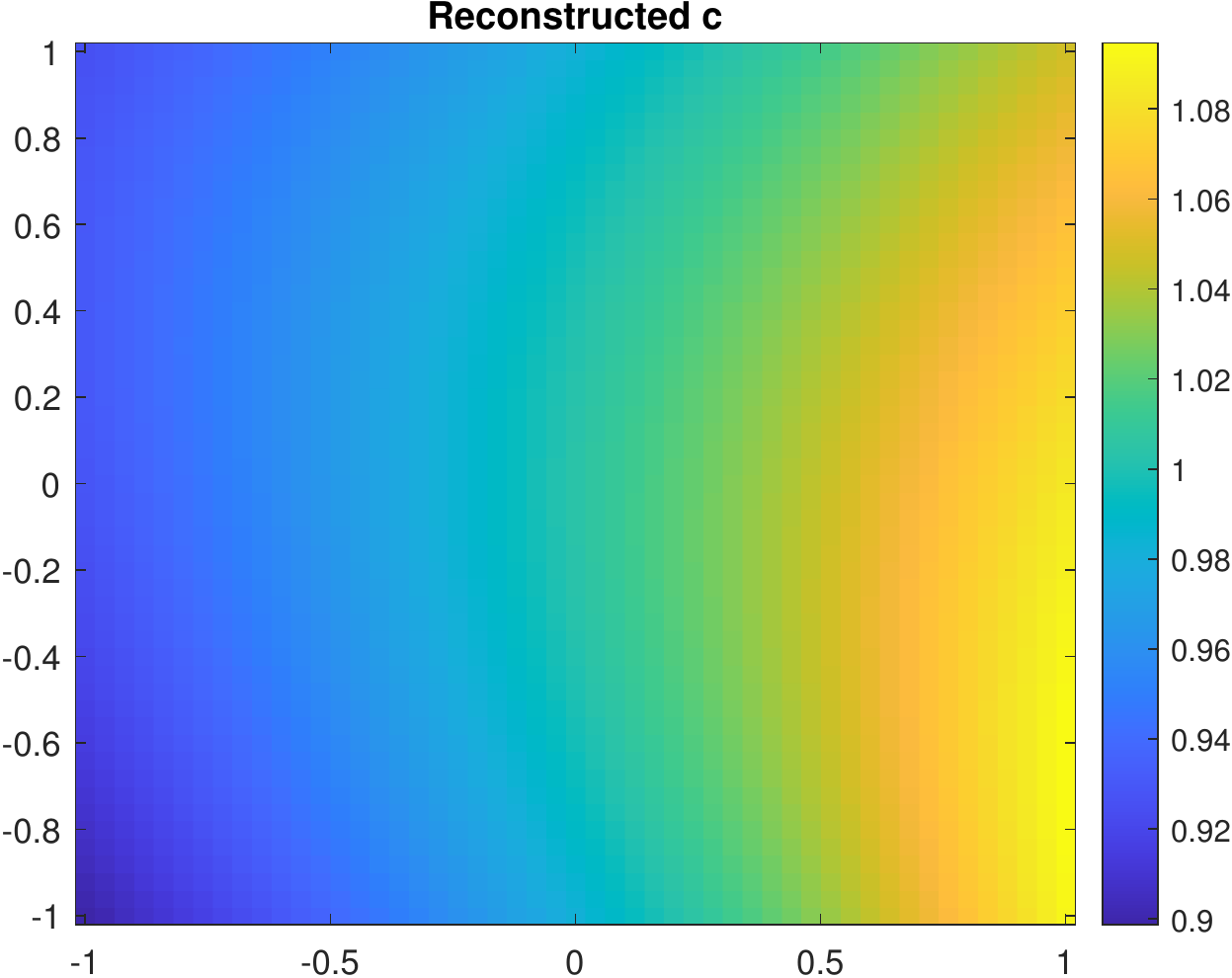}
    \includegraphics[width=0.49\textwidth]{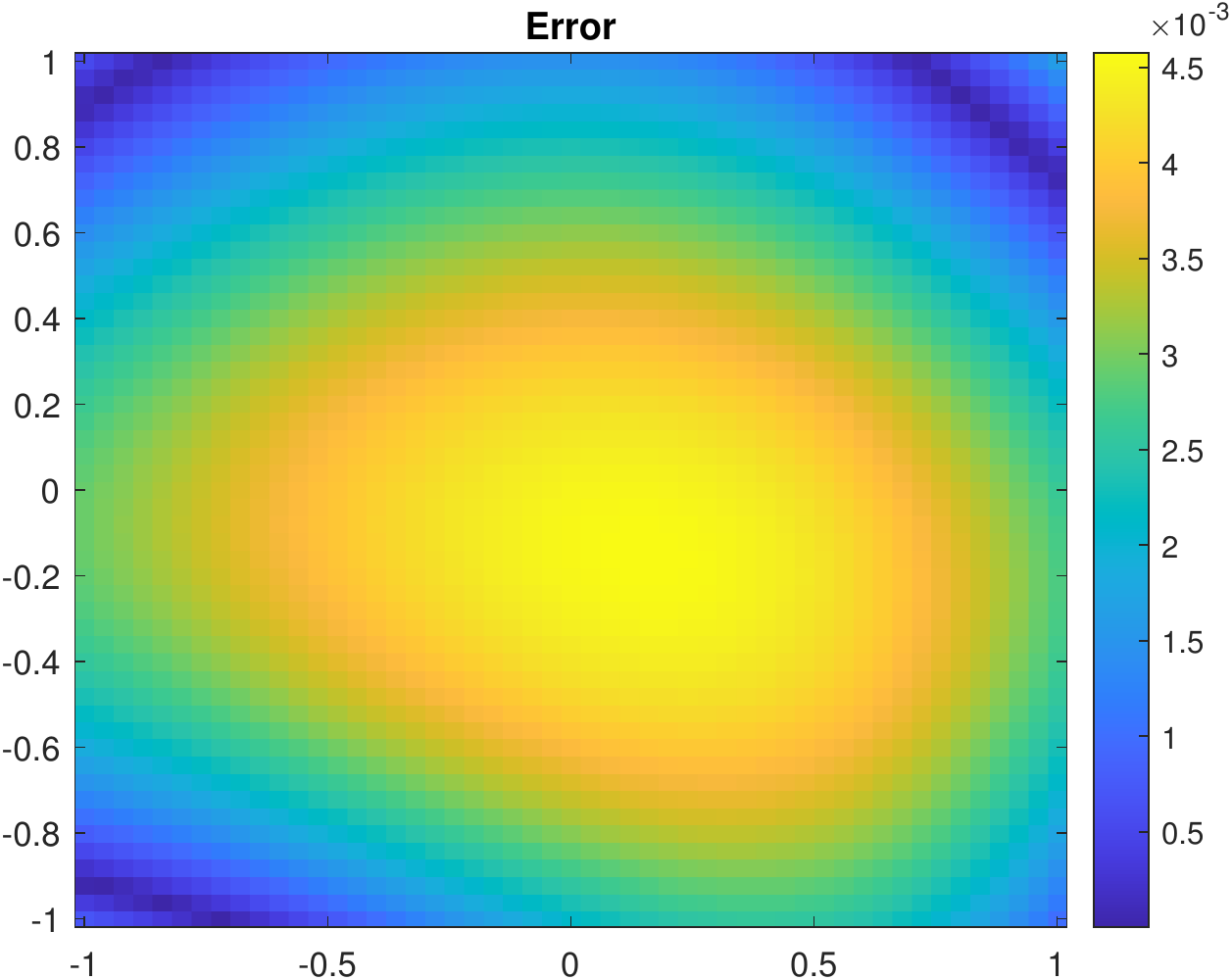}
    
    \includegraphics[width=0.49\textwidth]{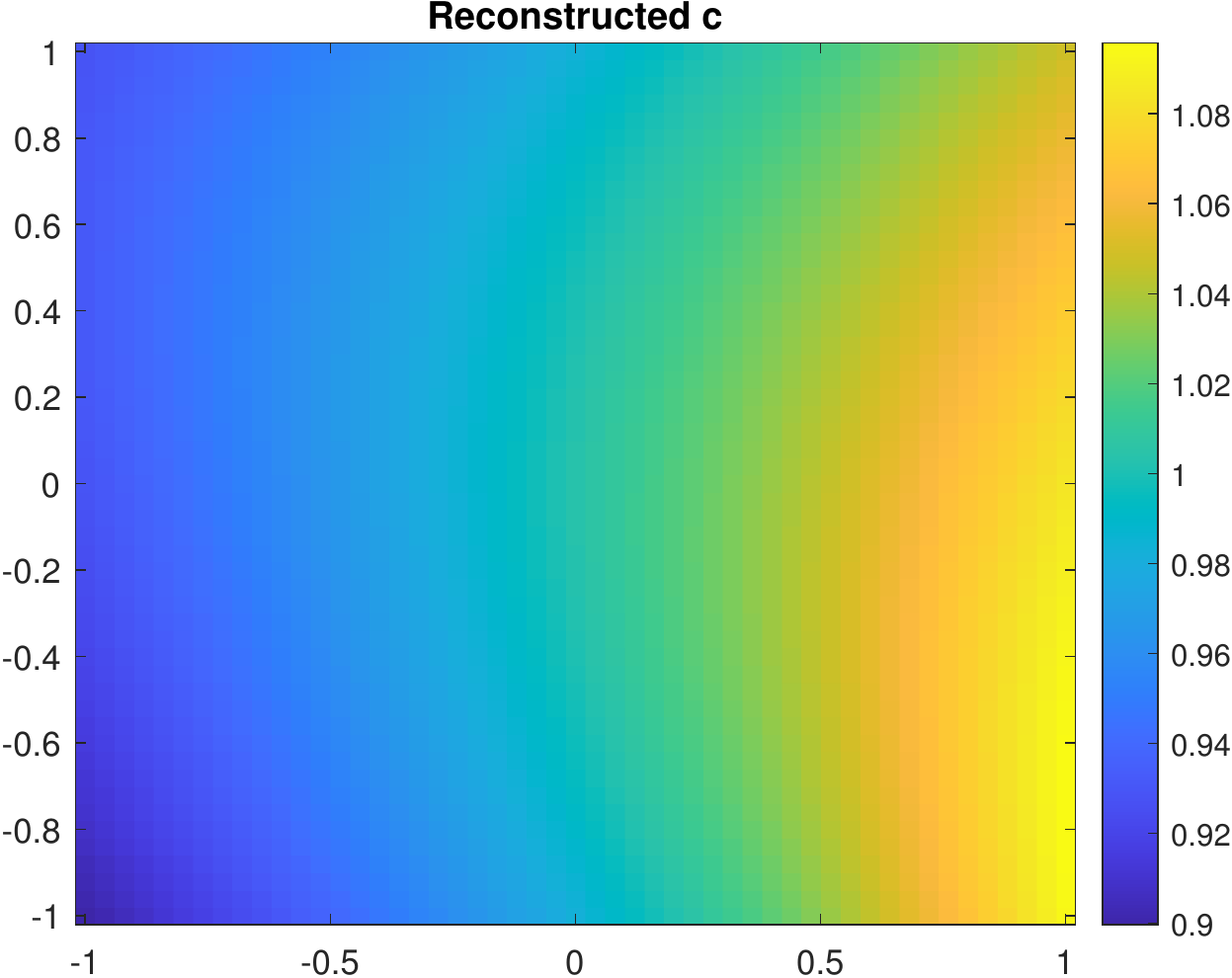}
    \includegraphics[width=0.49\textwidth]{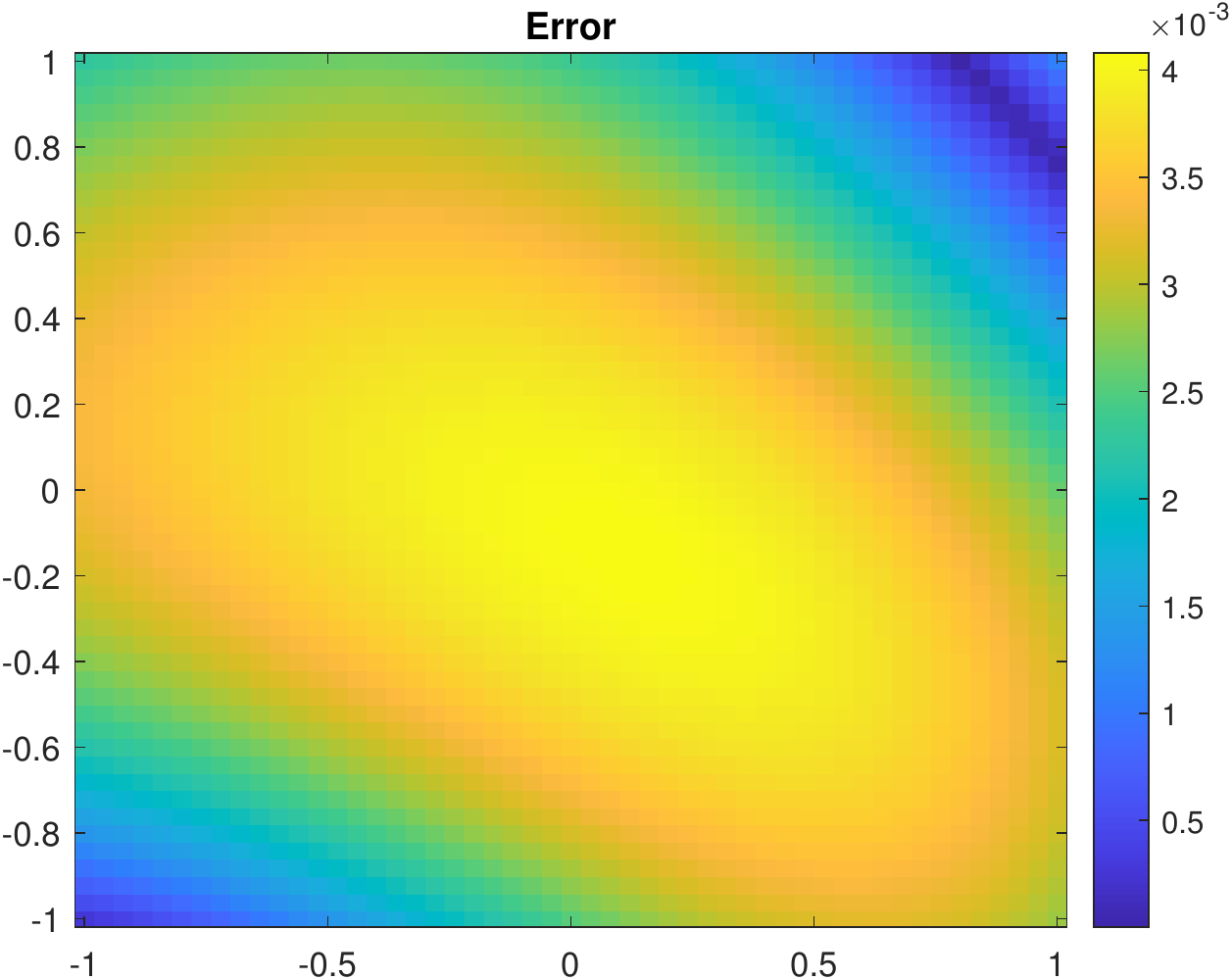}
    \caption{
    Reconstructions of the variable speed $c(x,y) = 1+0.08\sin{\pi x}+0.06\cos{\pi y}$.
    Left column: reconstructed $c$. 
    Right column: error between the reconstruction and the orthogonal projection of the ground truth. 
    First row: $0\%$ noise; the relative $L^2$-error is $0.3144\%$.
    Second row: $5\%$ noise; the relative $L^2$-error is $0.3153\%$.
    Third row: $50\%$ noise; the relative $L^2$-error is $0.3231\%$.
    Grid: $323\times51\times51$, $I=50,L=322$. 
    }
    \label{fig:exp2}
\end{figure}

\bigskip \bigskip
\textbf{Experiment 3: partial data.}


We test the algorithm with only partial knowledge of the ND map.
We use the constant speed $c=1$, see Figure~\ref{fig:exp1}, although the variable speed in Experiment 2 works almost equally well. Recall that the computational domain $\Omega$ is a square with four sides $x=\pm 1$ and $y=\pm 1$. We remove the knowledge of the ND map from the three sides $y=-1$, $x=1$, $y=1$ one after another. The reconstructions are shown in Figure~\ref{fig:exp3}, where the algorithm performs quite well. This is due to the large stoppage $T$ we choose.
No noise is imposed in this experiment.

\begin{figure}[p]
    \centering
    \includegraphics[width=0.49\textwidth]{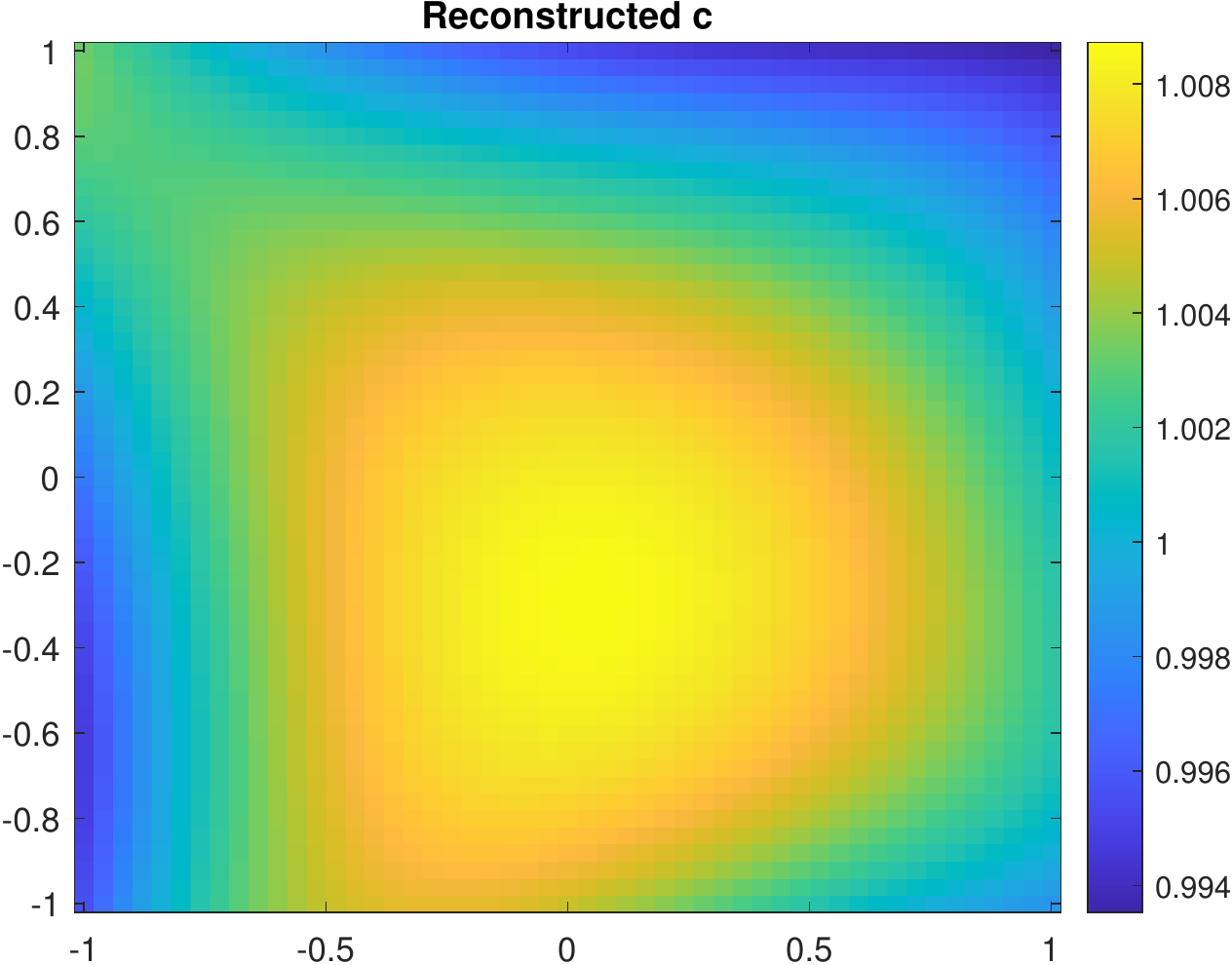}
    \includegraphics[width=0.49\textwidth]{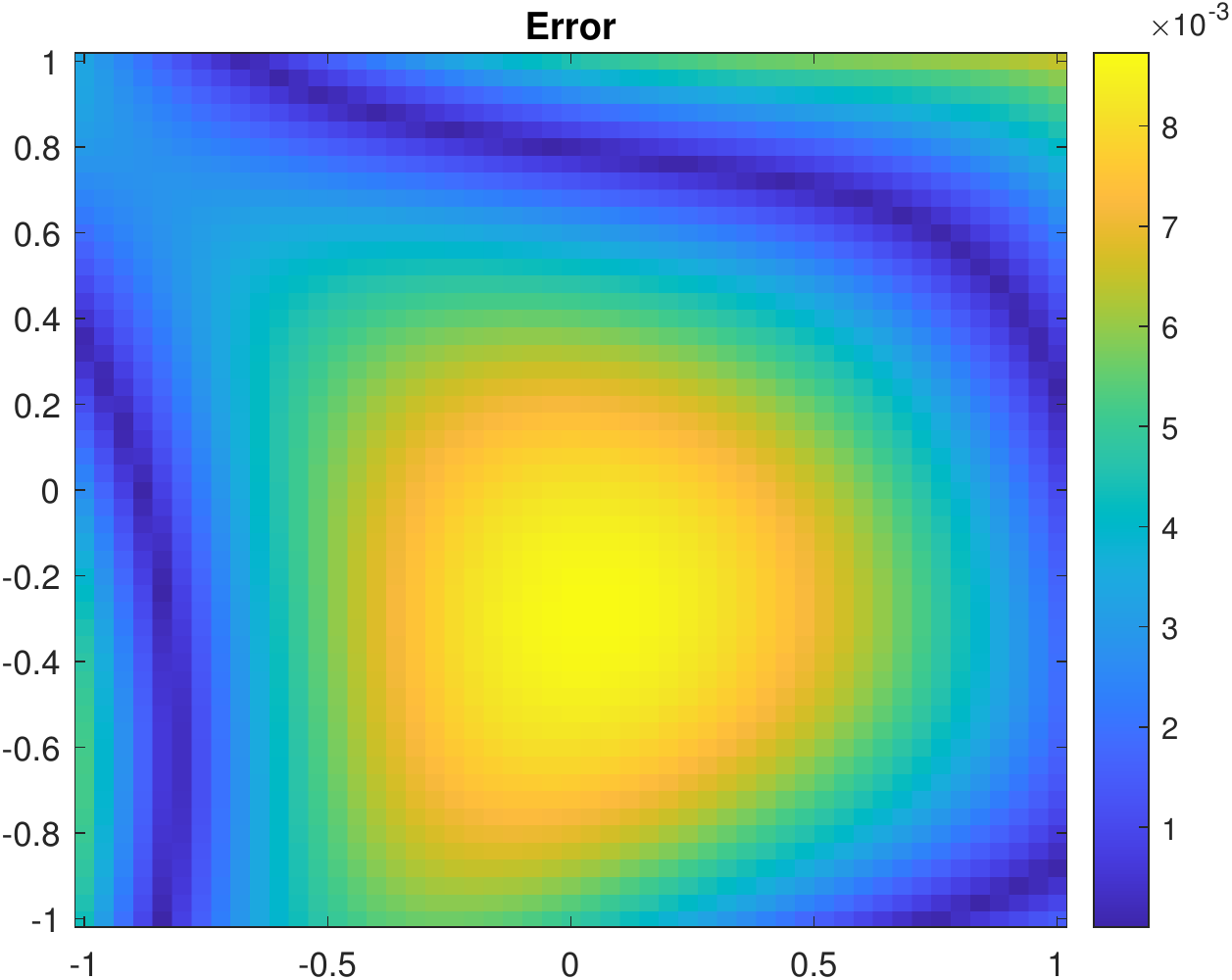}
    
    \includegraphics[width=0.49\textwidth]{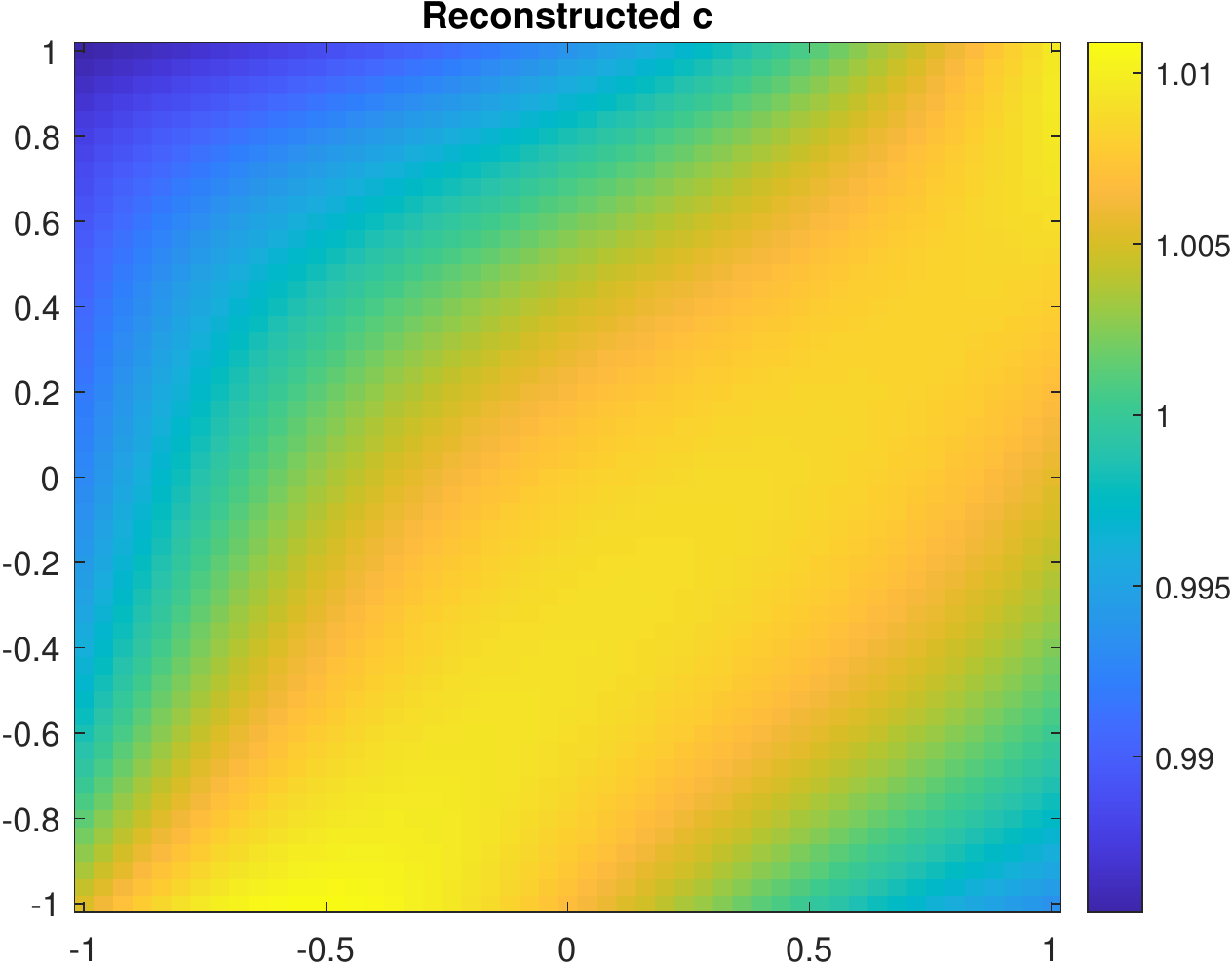}
    \includegraphics[width=0.49\textwidth]{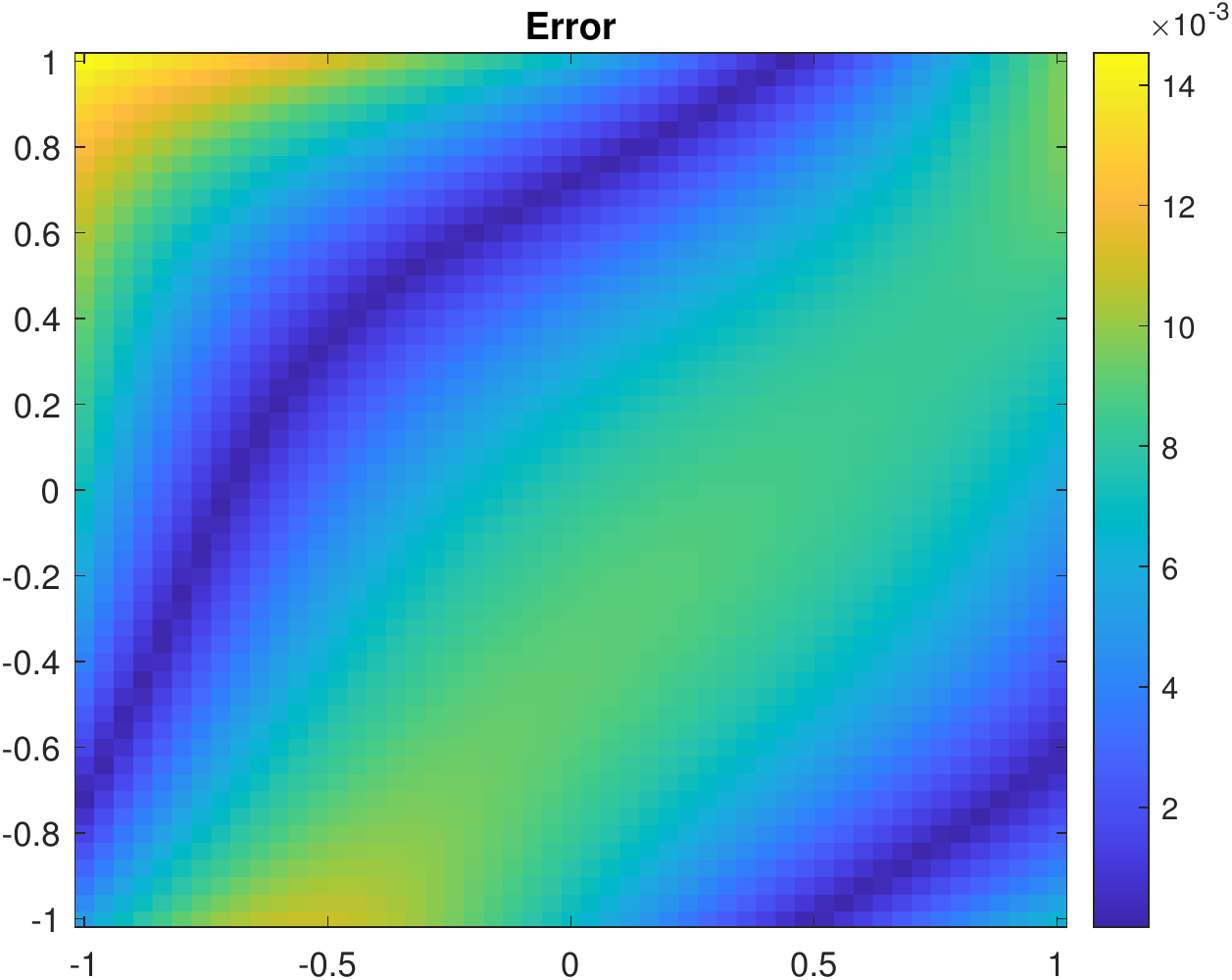}
    
    \includegraphics[width=0.49\textwidth]{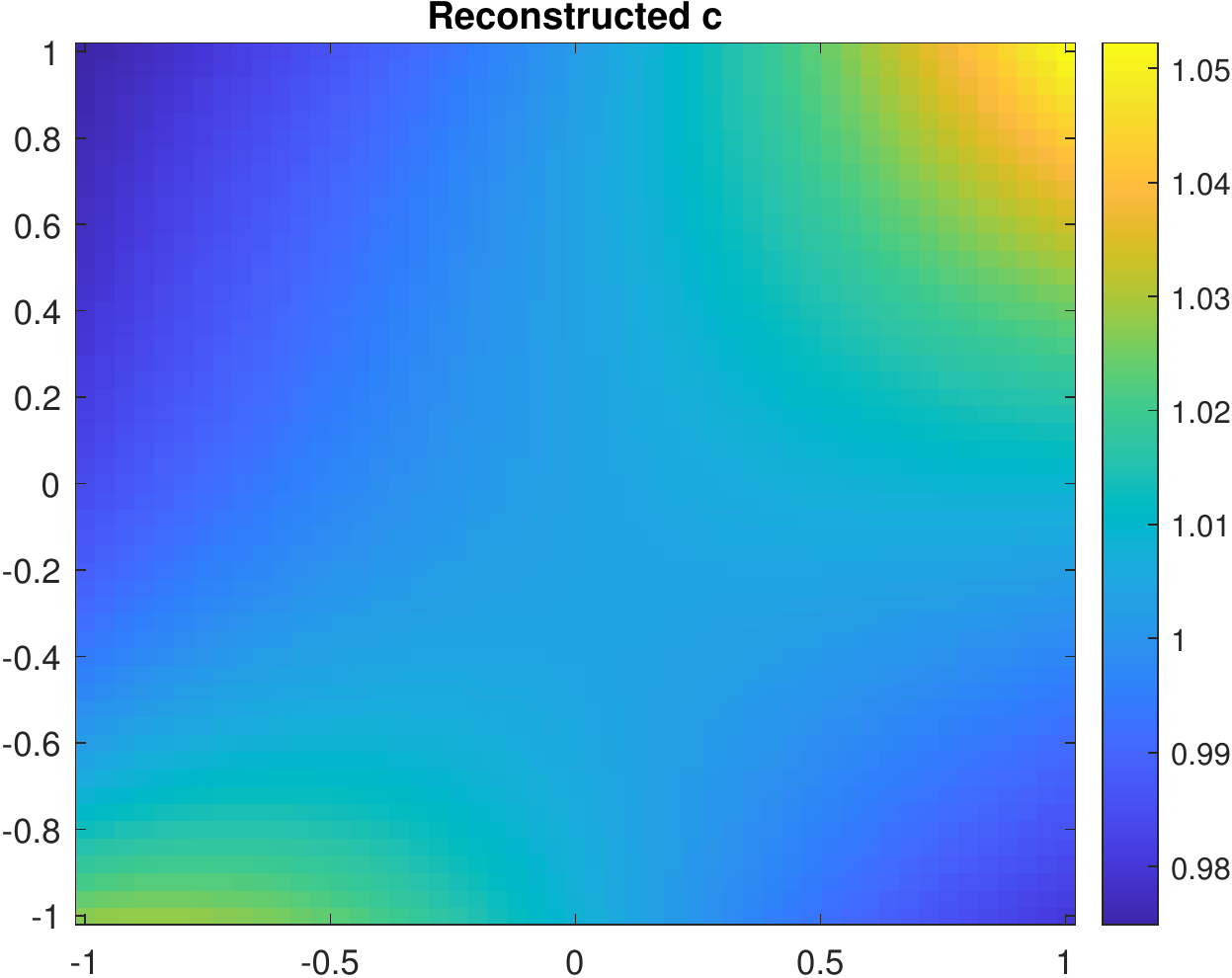}
    \includegraphics[width=0.49\textwidth]{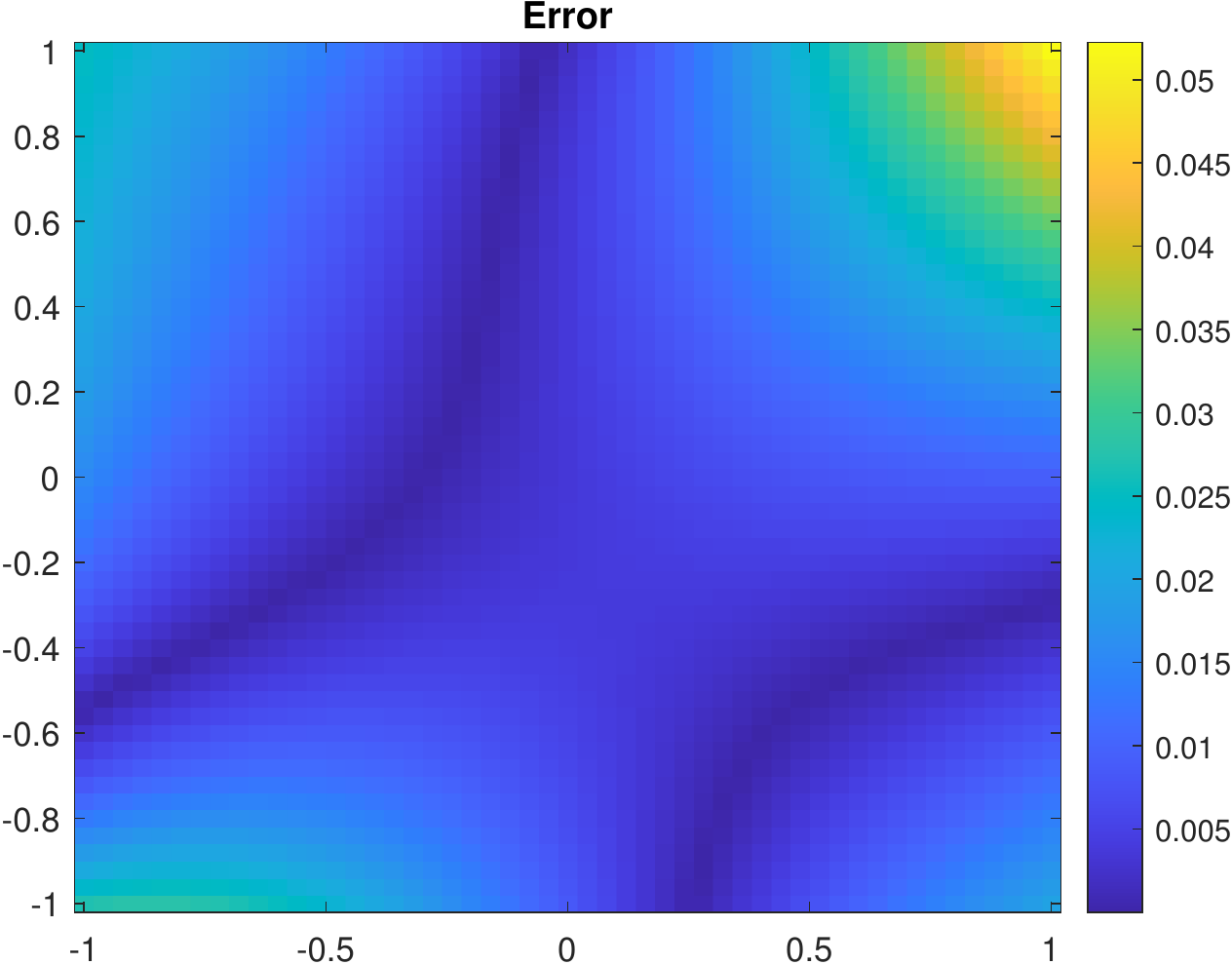}
    \caption{
    Left column: reconstructed $c$. 
    Right column: error between the reconstruction and the ground truth.
    First row: no data on $y=-1$; the relative $L^2$-error is $0.4954\%$.
    Second row: no data on $y=-1$ and $x=1$; the relative $L^2$-error is $0.6583\%$.
    Third row: no data on $y=\pm 1$ and $x=1$; the relative $L^2$-error is $1.2518\%$.
    Grid: $283\times51\times51$, $I=50,L=282$.
    }
    \label{fig:exp3}
\end{figure}

\bigskip \bigskip
\textbf{Experiment 4: $c$ is discontinuous}

This case is not covered by the theory, as Algorithm~\ref{alg:main} is derived under the assumption that $c$ is smooth. We still test it anyway. The wave speed is 
$$
    c(x,y)=\begin{cases}
    1&(x,y)\in[-0.5,0.5]^2,\\
    0.5&(x,y)\in\Omega\setminus[-0.5,0.5]^2,
    \end{cases}
$$
see Figure~\ref{fig:exp4_c}.
Again, the algorithm is able to reconstruct only the orthogonal projection of the discontinuous speed on $S_6$. However, this project is smooth and does not look like the original discontinuous speed, see Figure~\ref{fig:exp4}. 
No noise is imposed in this experiment.

\begin{figure}[!htb]
    \centering
    \includegraphics[width=0.49\textwidth]{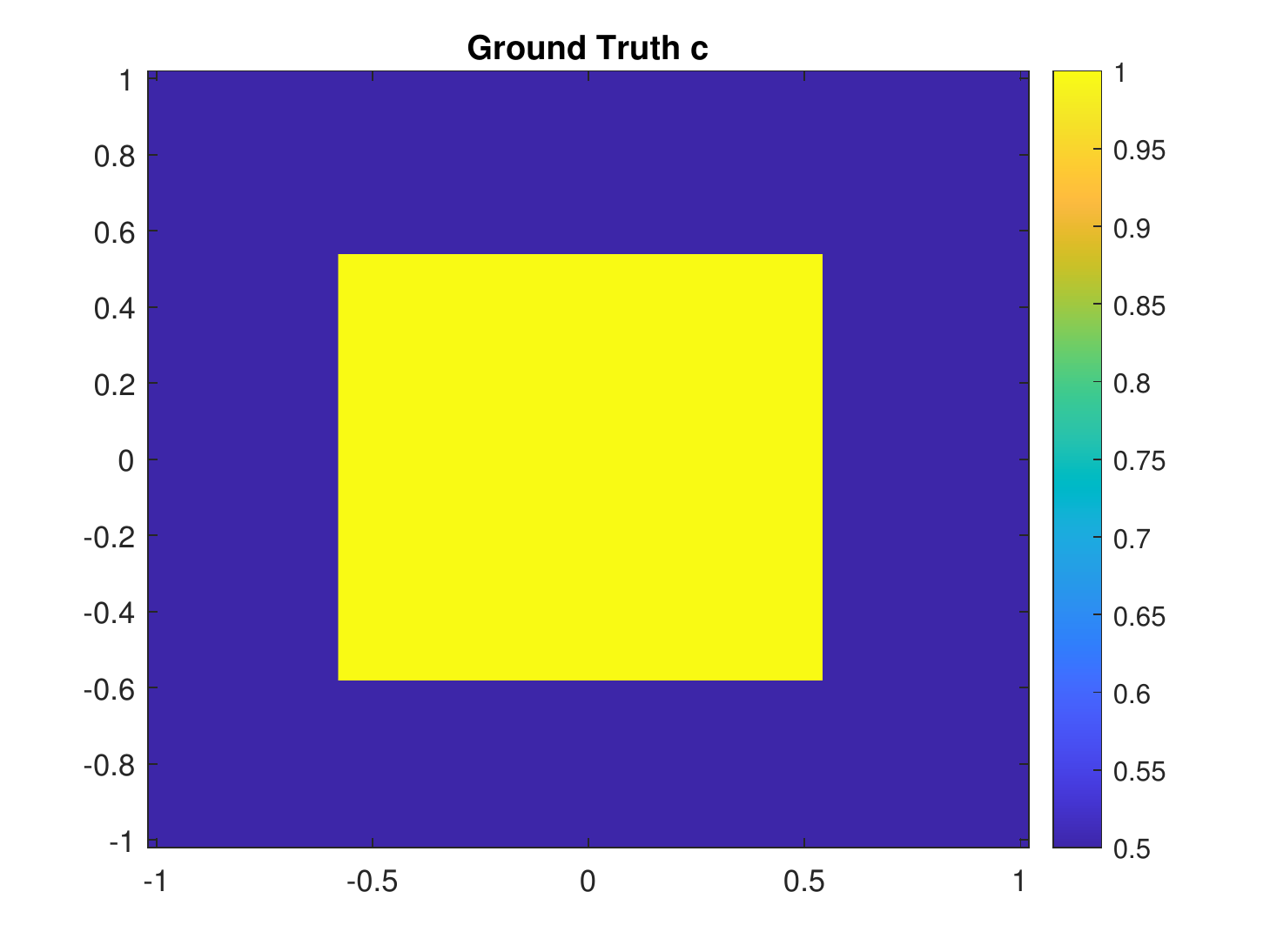}
    \includegraphics[width=0.49\textwidth]{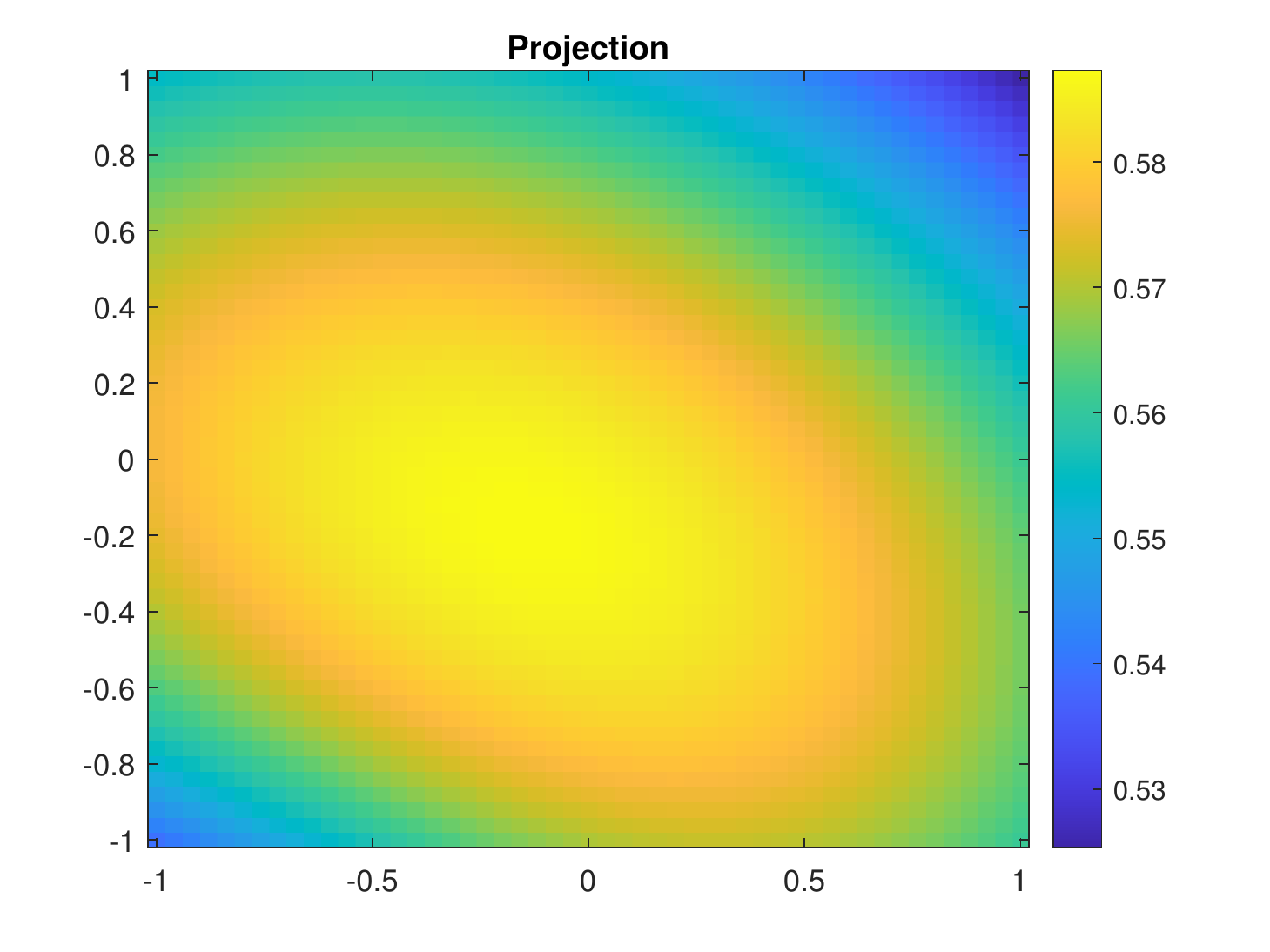}
    \caption{The ground truth speed $c$ and its projection.}
    \label{fig:exp4_c}
\end{figure}
\begin{figure}[!htb]
    \centering
    \includegraphics[width=0.49\textwidth]{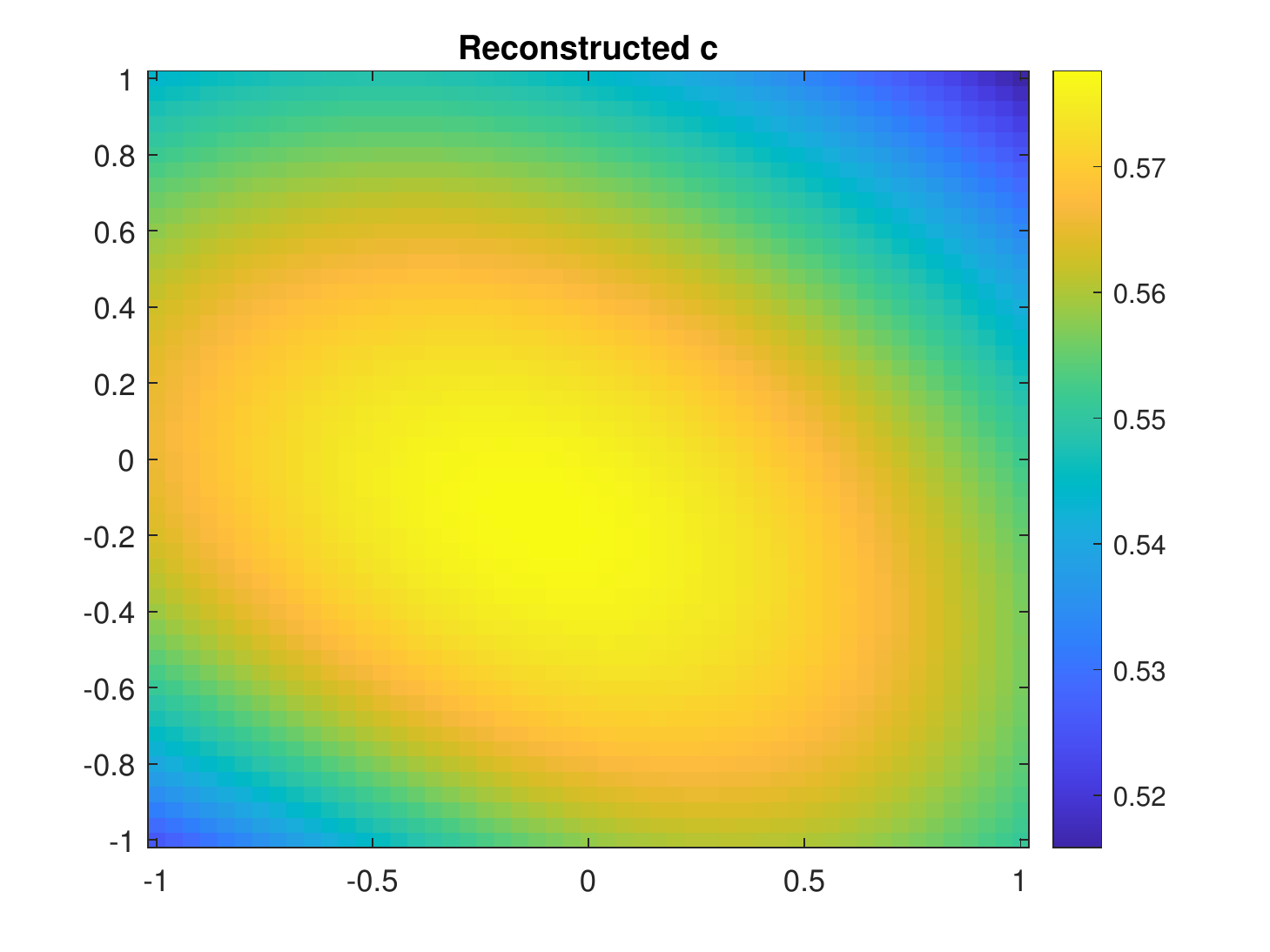}
    \includegraphics[width=0.49\textwidth]{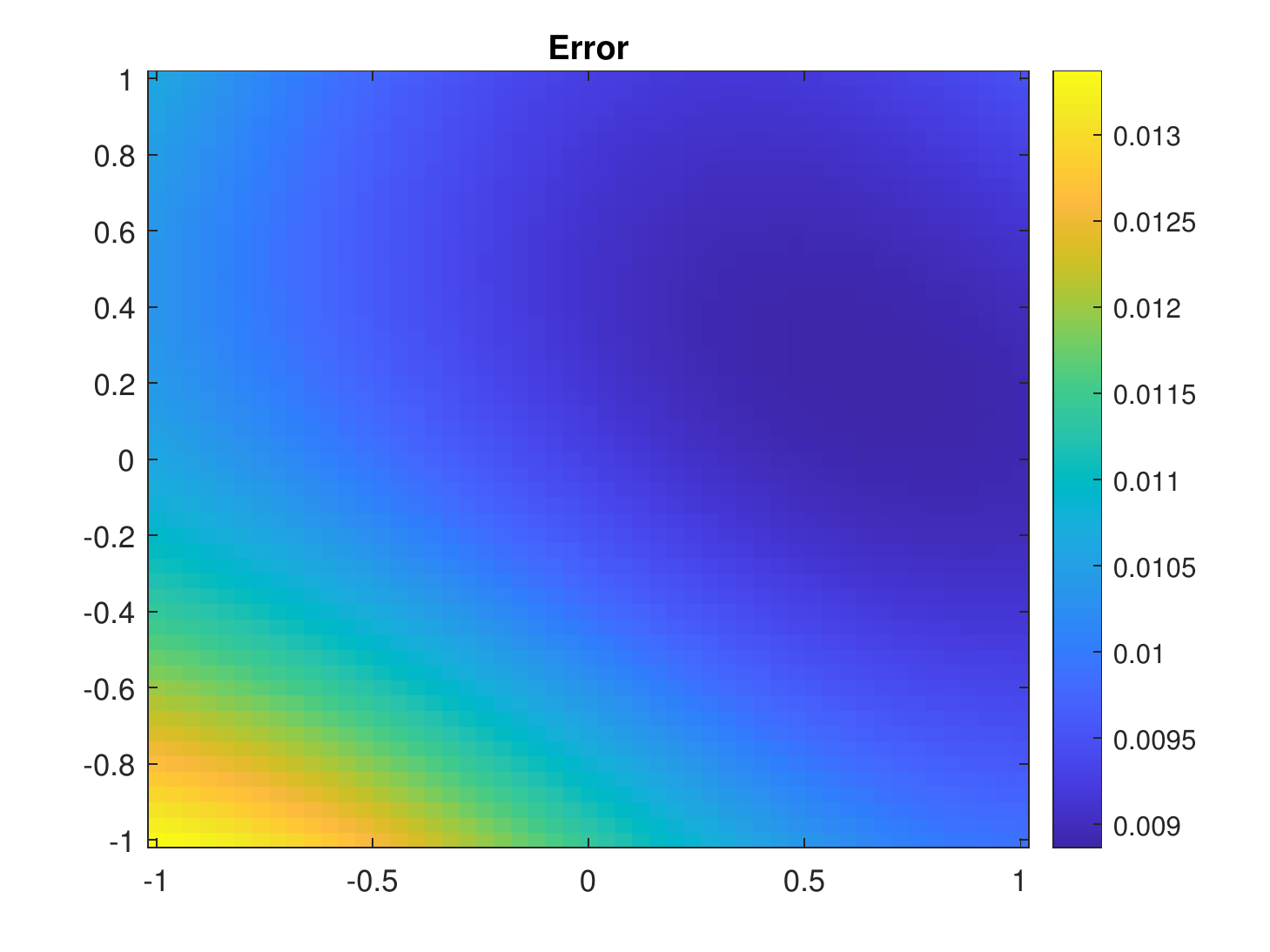}
    \caption{
    Left: reconstructed $c$.
    Right: error between the reconstruction and the orthogonal projection of the ground truth. The relative $L^2$ error is $1.7289\%$.
    Grid: $283\times51\times51$, $I=50,L=282$. 
    }
    \label{fig:exp4}
\end{figure}

\newpage

\section*{Acknowledgement}
The authors are very grateful to Dr. Lauri Oksanen for bringing their attention to this problem, and for communicating many references. 
The research of T. Yang and Y. Yang is partially supported by the NSF grant DMS-1715178, DMS-2006881, and the start-up fund from Michigan State University.

\bibliographystyle{abbrv}
\bibliography{refs}

\end{document}